\documentclass[a4paper,12pt]{article}
\usepackage[top=2.5cm,bottom=2.5cm,left=2.5cm,right=2.5cm]{geometry}
\usepackage[margin=1cm,%
font=small,%
format=hang,%
labelsep=period,%
labelfont=bf]{caption}
\usepackage[utf8]{inputenc}
\usepackage{amsmath}
\usepackage{amsfonts}
\usepackage{amssymb}
\usepackage{cite}
\usepackage{amsthm}
\usepackage{makeidx}
\usepackage{graphicx}
\usepackage{mathrsfs,xcolor}
\usepackage{blkarray}
\usepackage{bm}
\usepackage{setspace}
\usepackage[left,pagewise,displaymath, mathlines]{lineno}
\usepackage{enumitem}
\usepackage{multirow}
\usepackage{hhline}

\makeatletter
\newcommand\xleftrightarrow[2][]{\ext@arrow 0099{\longleftrightarrowfill@}{#1}{#2}}
\def\longleftrightarrowfill@{\arrowfill@\leftarrow\relbar\rightarrow}
\makeatother

\numberwithin{equation}{section}

\theoremstyle{plain}
\newtheorem{theorem}{Theorem}[section]
\newtheorem{lemma}{Lemma}[section]
\newtheorem{proposition}{Proposition}[section]
\newtheorem{corollary}{Corollary}[section]

\theoremstyle{definition}
\newtheorem{definition}{Definition}[section]

\newtheorem{remark}{Remark}[section]

\usepackage{authblk}

\author[1,*]{ \textbf{Noel T. Fortun}}\author[1,2,3]{\textbf{Eduardo R. Mendoza}}

\affil[1]{\footnotesize \textit{Department of Mathematics and Statistics, De La Salle University, Manila  0922, Philippines}}
\affil[2]{\textit{Center for Natural Sciences and Environmental Research, De La Salle University, Manila  0922, Philippines}}
\affil[3]{\textit{Max Planck Institute of Biochemistry, Martinsried near Munich, Germany}}
\affil[*]{Corresponding author: \texttt{noel.fortun@dlsu.edu.ph}}

\title{\textbf{Comparative analysis of carbon cycle models via kinetic representations}}

\date{}

\begin{document}
\maketitle
\thispagestyle{empty}
\begin{abstract}
The pre-industrial state of the global carbon cycle is a significant aspect of studies related to climate change. In this paper, we recall the power law kinetic representations of the pre-industrial models of Schmitz (2002) and Anderies et al. (2013) from our earlier work. The power law kinetic representations, as uniform formalism, allow for a more extensive analysis and comparison of the different models for the same system. Using the mathematical theories of chemical reaction networks (with power-law kinetics), this work extends the analysis of the kinetic representations of the two models and assesses the similarities and differences in their structural and dynamic properties in relation to model construction assumptions. The analysis includes but is not limited to the coincidence of kinetic and stoichiometric spaces of the networks, capacity for equilibria multiplicity and co-multiplicity, and absolute concentration robustness in some species. We bring together previously published results about the power law kinetic representations of the two models and consolidate them with new observations here. We also illustrate how the pre-industrial model of Anderies et al. may serve as a building block in the analysis of a kinetic representation of a global carbon cycle with carbon dioxide removal intervention.
\end{abstract}

\singlespacing

\section{Introduction}\label{sec1}

The pre-industrial state of the global carbon cycle is an important reference point for studies on climate change. Mathematical descriptions of this process were derived by Fortun et al.  for the model of R. Schmitz \cite{SCHM2002} and the model of Anderies et al. \cite{AND2013} in the form of power law kinetic representations in \cite{FLRM2019} and \cite{DOA2018} respectively. Power law kinetic representations are chemical reaction networks (CRN) with power law kinetics whose ODE system solutions very closely approximate those of the models. They are called kinetic realizations if the ODE systems coincide, signifying the dynamical equivalence of the systems. The use of a uniform formalism (such as power law kinetics) enables deeper analysis and comparison of different models for the same system. The goal of this paper is twofold: 
\begin{enumerate}[label=(\roman*)]
    \item extend the analysis of the kinetic representations of the Schmitz and Anderies models; and
    \item assess the coincidence and difference in their structural and dynamic properties in relation to model construction assumptions.
\end{enumerate}
For notational brevity, we will refer to Schmitz and Anderies kinetic representations as \textit{Schmitz} and \textit{Anderies systems} respectively.

In addition to the proof of the existence of positive equilibria for any Schmitz system in \cite{FLRM2019}, previous results include the construction by Nazareno et al. \cite{NEML2019} of a linear conjugate with interesting properties (s. Section \ref{sec4.4}), the analysis by Lao et al. \cite{LLMM2022} of absolute concentration robustness (ACR) in a special subsystem (s. Section \ref{sec4.2}) and the derivation of the log-parametrization property for the same subsystem by Hernandez and Mendoza in \cite{HEME2022b}. In \cite{DOA2018}, an Anderies system was shown to have the capacity for multistationarity, i.e., the occurrence of distinct positive equilibria in a stoichiometric class. The same authors constructed in \cite{FLRM2021} another Anderies system which displayed absolute concentration robustness in two of its three species.

The new results of this paper for both Schmitz and Anderies systems are:
\begin{itemize}
    \item coincidence of the kinetic and stoichiometric subspaces (which implies e.g., the invariance of interesting network properties under linear conjugacy);
    \item (exponential) stability of the positive equilibria; and
    \item availability of a weakly reversible ``low deficiency complement" (LDC), i.e., a linear conjugate system whose deficiency is $1-\delta$, where $\delta$ is the network's deficiency, and whose kinetics is PL-RDK.
\end{itemize}
Furthermore, the paper derives the following novel class-specific properties:
\begin{itemize}
    \item Schmitz systems are Birch systems, i.e., there is a unique positive equilibrium in each stoichiometric class and it is complex balanced (s. Section \ref{sec4.1});
    \item Anderies systems are PLP systems, i.e., the set of positive equilibria has the form $\{ x\in \mathbb{R}^\mathscr{S}_{>0} \mid \log x - \log x^* \in (P_E)^\perp \}$, where $P_E$ is a subspace of $\mathbb{R}^\mathscr{S}$ and $x^*$ is a positive equilibrium, which leads to the identification of three distinct classes $\textsf{AND}_<$, $\textsf{AND}_0$, and $\textsf{AND}_>$ (s. Section \ref{sec5.3});
    \item absence of species with ACR in Schmitz, $\textsf{AND}_<$ and $\textsf{AND}_>$ systems vs. occurrence of two ACR species in $\textsf{AND}_0$; and
    \item mono- and co-monostationarity in Schmitz and $\textsf{AND}_0$ systems vs. multi- and co-multistationarity in $\textsf{AND}_>$ systems. 
\end{itemize}
Several of the results above derive from more general propositions about conservative, closed kinetic systems of maximal rank, i.e., systems whose stoichiometric subspaces are hyperplanes, of which both Schmitz and Anderies systems are examples.

Moreover, a kinetic representation of an aggregated Schmitz model, i.e., the set of species is reduced to coincide with those of the Anderies model, is shown to have the same structural and kinetic properties (with one exception) as a Schmitz system. Comparison of an aggregated Schmitz system with the (dynamically equivalent) LDC of an Anderies system reveals differences in only three structural and three kinetic properties. These may be viewed as the essential properties resulting from the different hypotheses underlying the Schmitz and Anderies et al. models.

Finally, this paper provides an analysis of a  kinetic representation of a model of carbon dioxide removal (CDR) from Heck et al. \cite{HECK2016}. The motivation for this analysis comes from the observation that two of its subsystems are (structurally equivalent to) pre-industrial Anderies systems.

The paper is organized as follows: Section \ref{sec2} collects fundamental concepts and results on reaction networks and kinetic systems needed in subsequent sections. A review of the Schmitz and Anderies et al. models is provided in Section \ref{sec3}. Sections \ref{sec4} and \ref{sec5} first derive the new results for Schmitz and Anderies systems respectively and then combine them with previous results (in tables) to present an overview of similarities and differences. In Section \ref{sec6}, aggregated Schmitz systems are constructed and then compared with the Anderies systems. Section \ref{sec7} covers a reaction network-based analysis of a CDR model using the Anderies pre-industrial model as a basis. A summary and outline of perspectives for further research are presented in Section \ref{sec8}.

\section{Preliminaries}\label{sec2}
In this Section, we assemble important notions and necessary results on chemical reaction networks and chemical kinetic systems to establish a foundation for the succeeding sections. In general, this paper uses the standard nomenclature in chemical reaction network theory (CRNT) \cite{FEIN1979,FEIN2019,TOTH2018}. For a list of frequently used symbols and abbreviations, the reader may refer to Appendix \ref{append:nomenclature}. 

\subsubsection*{Notation}
We denote the real numbers by $\mathbb{R}$, the non-negative real numbers by $\mathbb{R}_{\geq0}$ and the positive real numbers by $\mathbb{R}_{>0}$.  Objects in reaction systems are viewed as members of vector spaces. Suppose $\mathscr{I}$ is a finite index set. By $\mathbb{R}^\mathscr{I}$, we mean the usual vector space of real-valued functions with domain $\mathscr{I}$.  If $x \in \mathbb{R}_{>0}^\mathscr{I}$ and $y \in \mathbb{R}^\mathscr{I}$, we define $x^y \in \mathbb{R}_{>0}$ by
$
x^y= \prod_{i \in \mathscr{I}} x_i^{y_i} .
$
Let $x \wedge y$ be the component-wise minimum, $(x \wedge y)_i = \min (x_i, y_i)$.
The vector $\log x\in \mathbb{R}^\mathscr{I}$,where $x \in \mathbb{R}_{>0}^\mathscr{I}$, is given by 
$(\log x)_i = \log x_i,  \text{ for all } i \in \mathscr{I}.$  If $x,y \in  \mathbb{R}^\mathscr{I}$, the standard scalar product $ \langle x, y \rangle \in  \mathbb{R}$ is defined by $ \langle x, y \rangle = \sum_{i \in \mathscr{I}} x_i y_i.$ 
The \textit{support} of $x \in \mathbb{R}^\mathscr{I}$, denoted by $\text{supp } x$, is given by
$ \text{supp } x := \{ i \in \mathscr{I} \mid x_i \neq 0 \}.$

\subsection{Fundamentals of chemical reaction networks}

We begin with the formal definition of a chemical reaction network or CRN. 

\begin{definition}
A \textbf{chemical reaction network} or CRN is a triple $\mathscr{N}:= (\mathscr{S,C,R})$ of nonempty finite sets $\mathscr{S}$, $\mathscr{C}$, and $\mathscr{R}$, of $m$ \textbf{species}, $n$ \textbf{complexes}, and $r$ \textbf{reactions}, respectively, where $\mathscr{C} \subseteq \mathbb{R}_{\geq 0}^\mathscr{S}$ and $\mathscr{R} \subset \mathscr{C} \times \mathscr{C}$ satisfying the following properties:
\begin{enumerate}[label=(\roman*)]
    \item $(y,y) \notin \mathscr{R}$ for any $y \in \mathscr{C}$;
    \item for each $y \in \mathscr{C}$, there exists $y' \in \mathscr{C}$ such that $(y,y')\in \mathscr{R}$ or $(y',y)\in \mathscr{R}$.
\end{enumerate}
\end{definition}
\noindent For $y \in \mathscr{C}$, the vector $$y=\displaystyle{\sum_{S \in \mathscr{S}}} y_S S,$$ where $y_S$ is the \textbf{stoichiometric coefficient} of the species $S$. In lieu of $(y,y')\in \mathscr{R}$, we write the more suggestive notation  $y \rightarrow y'$. In this reaction, the vector $y$ is called the \textbf{reactant complex} and $y'$ is called the \textbf{product complex}. 

CRNs can be viewed as directed graphs where the complexes are vertices and the reactions are arcs. The (strongly) connected components are precisely the \textbf{(strong) linkage classes} of the CRN. A strong linkage class is a \textbf{terminal strong linkage class} if there is no reaction from a complex in the strong linkage class to a complex outside the given strong linkage class. 

\begin{definition}
A CRN with $n$ complexes, $n_r$ reactant complexes, $\ell$ linkage classes, $s\ell$ strong linkage classes, and $t$ terminal strong linkage classes is 
\begin{enumerate}[label=(\roman*)]
    \item \textbf{weakly reversible} if $s\ell = \ell$;
    \item $\bm{t}$\textbf{-minimal} if $t=\ell$;
    \item \textbf{point terminal} if $t=n-n_r$; and
    \item \textbf{cycle terminal} if $n-n_r=0$.
\end{enumerate}
\end{definition}

For every reaction, we associate a \textbf{reaction vector}, which is obtained by subtracting the reactant complex from the product complex. From a dynamic perspective, every reaction $ y \rightarrow y' \in \mathscr{R}$ leads to a change in species concentrations proportional to the  reaction vector $ \left( y' – y \right) \in \mathbb{R}^\mathscr{S}$. The overall change induced by all the reactions lies in a subspace of $\mathbb{R}^\mathscr{S}$ such that any trajectory in $\mathbb{R}^\mathscr{S}_{>0}$ lies in a coset of this subspace. 

\begin{definition}
The \textbf{stoichiometric subspace} of a network $\mathscr{N}$ is given by
$$ \mathcal{S} := \text{span } \{ y' – y \in \mathbb{R}^\mathscr{S} \mid y \rightarrow y' \in \mathscr{R} \}.$$
The \textbf{rank} of the network is defined as $s:= \dim \mathcal{S}$. For $x \in \mathbb{R}^\mathscr{S}_{>0}$, its \textbf{stoichiometric compatibility class} is defined as $(x+\mathcal{S}) \cap \mathbb{R}^\mathscr{S}_{ \geq 0}$. Two vectors $x^{*}, x^{**} \in  \mathbb{R}^\mathscr{S}$ are \textbf{stoichiometrically compatible} if $ x^{**}-x^{*} \in \mathcal{S}$.
\end{definition}

An important structural index of a CRN, called \textit{deficiency}, provides one way to classify networks.

\begin{definition}
The \textbf{deficiency} $\delta$ of a CRN with $n$ complexes, $\ell$ linkage classes, and rank $s$ is defined as $\delta:=n-\ell-s$.
\end{definition}

\subsection{Fundamentals of chemical kinetic systems}
It is generally assumed that the rate of a reaction $y \rightarrow y' \in \mathscr{R}$ depends on the concentrations of the species in the reaction. The exact form of the non-negative real-valued rate function $K_{ y \rightarrow y'}$ depends on the underlying \textit{kinetics}. 

\subsubsection{General Kinetics}
The following definition of kinetics is expressed in a more general context than what one typically finds in CRNT literature.
\begin{definition}
A \textbf{kinetics} for a network $\mathscr{N}=(\mathscr{S,C,R})$ is an assignment to each reaction $y \rightarrow y' \in \mathscr{R}$ a rate function $ K_{ y \rightarrow y'}: \Omega_K \rightarrow \mathbb{R}_{\geq 0}$, where $\Omega_K$ is a set such that $\mathbb{R}^\mathscr{S}_{> 0} \subseteq \Omega_K \subseteq \mathbb{R}^\mathscr{S}_{\geq 0}$, $x  \wedge  x^{*} \in \Omega_K$ whenever $x, x^{*} \in \Omega_K$, and $ K_{ y \rightarrow y'} (x) \geq 0$ for all $x \in \Omega_K$. A kinetics for a network $\mathscr{N}$ is denoted by $K:\Omega_K \rightarrow \mathbb{R}^\mathscr{R}_{\geq 0}$ (\cite{WIUF2013}). A \textbf{chemical kinetics} is a kinetics $K$ satisfying the condition that for each $y \rightarrow y' \in \mathscr{R}$, $ K_{ y \rightarrow y'} (x) >0$ if and only if $\text{supp } y \subset \text{supp } x$. The pair $(\mathscr{N},K)$ is called a \textbf{chemical kinetic system} (\cite{AJLM2017}). 
\end{definition}

The system of ordinary differential equations that govern the dynamics of a CRN is defined as follows.

\begin{definition}\label{def:ODE}
The \textbf{ordinary differential equation (ODE)} associated with a chemical kinetic system $(\mathscr{N},K)$ is defined as 
$ \dfrac{dx}{dt}=f(x)$ with \textbf{species formation rate function} 
\begin{equation}\label{eq:sfrf}
    f(x)= \sum_{ y \rightarrow y' \in \mathscr{R}} K_{ y \rightarrow y'} (x) (y'-y).
\end{equation}
A \textbf{positive equilibrium} or \textbf{steady state} $x$ is an element of $\mathbb{R}^\mathscr{S}_{>0}$ for which $f(x) = 0$.
\end{definition}

The reaction vectors of a CRN $ (\mathscr{S,C,R})$ are \textbf{positively dependent} if for each reaction $y \rightarrow y' \in \mathscr{R}$, there exists a positive number $k_{ y \rightarrow y'}$ such that $\sum_{y \rightarrow y' \in \mathscr{R}}k_{ y \rightarrow y'} (y'-y)=0$. In view of Definition \ref{def:ODE}, a necessary condition for a chemical kinetic system to admit a positive steady state is that its reaction vectors are positively dependent.

\begin{definition}
The \textbf{set of positive equilibria} or \textbf{steady states} of a chemical kinetic system $(\mathscr{N},K)$ is given by 
$$ E_+ (\mathscr{N},K) = \{ x \in \mathbb{R}^\mathscr{S}_{>0} \mid f(x) = 0 \}. $$
For brevity, we also denote this set by $E_+$. The chemical kinetic system is said to be \textbf{multistationary} (or has the capacity to admit \textbf{multiple steady states}) if there exist positive rate constants such that $\mid E_+ \cap \mathcal{P}\mid \geq 2$ for some positive stoichiometric compatibility class $\mathcal{P}$. On the other hand, it is \textbf{monostationary} if $\mid E_+ \cap \mathcal{P}\mid \leq 1$ for all positive stoichiometric compatibility class $\mathcal{P}$.
\end{definition}

To reformulate the species formation rate function in Eq. (\ref{eq:sfrf}), we consider the natural basis vectors $\omega_i \in \mathbb{R}^\mathscr{I}$ where $i \in \mathscr{I}=\mathscr{C}$ or $\mathscr{R}$ and define 
\begin{enumerate}[label=(\roman*)]
    \item the \textbf{molecularity map} $Y: \mathbb{R}^\mathscr{C} \rightarrow \mathbb{R}^\mathscr{S}$ with $Y(\omega_y)=y$;
    \item the \textbf{incidence map} $I_a: \mathbb{R}^\mathscr{R} \rightarrow \mathbb{R}^\mathscr{S}$ with $I_a (\omega_{y \rightarrow y'})= \omega_{y'} - \omega_y$; and
    \item the \textbf{stoichiometric map} $N: \mathbb{R}^\mathscr{R} \rightarrow \mathbb{R}^\mathscr{S}$ with $N= YI_a$.
\end{enumerate}
Hence, Eq. (\ref{eq:sfrf}) can be rewritten as $f(x)=YI_aK(x)=NK(x).$ The positive steady states of a chemical kinetic system that satisfies $I_a K(x)=0$ are called \textit{complex balancing equlibria}. 
\begin{definition}
The \textbf{set of complex balanced equilibria} of a chemical kinetic system $(\mathscr{N},K)$ is the set
$$ Z_+(\mathscr{N},K) = \{ x\in \mathbb{R}^\mathscr{S}_{>0} \mid I_a K(x) =0 \} \subseteq E_+(\mathscr{N},K).$$
A chemical kinetic system is said to be \textbf{complex balanced} if it has a complex balanced equilibrium. A complex balanced kinetic system is \textbf{absolutely complex balanced (ACB)} if every positive equilibrium is complex balanced.
\end{definition}

\begin{remark}\label{rem:CB}
If a chemical kinetic system has zero deficiency, then its steady states are all complex balanced (Corollary 4.8, \cite{FEIN1979}). Moreover, the existence of a complex balanced equilibrium implies that the underlying CRN is weakly reversible (Theorem 2B, \cite{HORN1972}). 
\end{remark}

\subsubsection{Power law kinetic systems} 
Power law kinetics generalize mass action kinetics. For systems where molecular overcrowding is observed, the kinetic orders for the reactions can exhibit non-integer values \cite{BAJZER2008,SAVA1998,Schnell2004} found in power-law formalism \cite{SAVA1969,SAVA1998,VOIT2000,VOIT2006,VOIT2013}. 

We define power law kinetics through the $r \times m$ \textbf{kinetic order matrix} $F=[F_{ij}]$, where $F_{ij} \in \mathbb{R}$ encodes the kinetic order the $j$th species of the reactant complex in the $i$th reaction. Further, consider the \textbf{rate vector} $k \in \mathbb{R}^\mathscr{R}_{>0}$, where $k_i \in \mathbb{R}_{>0}$ is the rate constant in the $i$th reaction. 

\begin{definition}\label{def:PLK}
A kinetics $K: \mathbb{R}^\mathscr{S}_{>0} \rightarrow \mathbb{R}^\mathscr{R}$ is a \textbf{power law kinetics} or \textbf{PLK} if
$$\displaystyle K_{i}(x)=k_i x^{F_{i,*}} \quad \text{for all } i \in \mathscr{R},$$
where $F_{i,*}$ is the row vector containing the kinetic orders of the species of the reactant complex in the $i$th reaction.
\end{definition}
 
Power law kinetic systems can be classified based on kinetic orders assigned to its \textbf{branching reactions}, i.e., reactions sharing a common reactant complex. 

\begin{definition}
A PLK system has \textbf{reactant-determined kinetics} (or of type \textbf{PL-RDK}) if for any two reactions $i$, $j \in \mathscr{R}$ with identical reactant complexes, the corresponding rows of kinetic orders in $F$ are identical, i.e. $F_{ih}=F_{jh}$ for all $h  \in \mathscr{S}$.  Otherwise, a PLK system has \textbf{non-reactant-determined kinetics} (or of type \textbf{PL-NDK}).
\end{definition}

\begin{remark}
In a mass action system where the reactions occur in a homogeneous space, the kinetic order is the same as the number of molecules entering into the reaction. Hence, in view of Definition \ref{def:PLK}, a kinetics is a \textbf{mass action kinetics} if the entries of the row vector $F_{i,*}$ are  the stoichiometric coefficients of a reactant complex in the $i$th reaction. Moreover, mass action kinetics is of type PL-RDK. 
\end{remark}

Arceo et al. \cite{ARCEO2015} identified two large sets of kinetic systems, namely the \textbf{complex factorizable (CF)} kinetics and its complement, the \textbf{non-complex factorizable (NF)} kinetics. Complex factorizable kinetics generalize the key structural property of mass action kinetics that the species formation rate function decomposes as $f(x) = Y \circ A_k \circ \Psi_k,$
where $Y$ is the map of complexes, the Laplacian map $A_k : \mathbb{R}^\mathscr{C} \rightarrow \mathbb{R}^\mathscr{C}$ defined by
$A_k x:= \sum_{y \rightarrow y'\in\mathscr{R}}k_{y \rightarrow y'}x_y (\omega_{y'} -\omega_y)$, and $\Psi_k: \mathbb{R}^\mathscr{S}_{\geq 0} \rightarrow  \mathbb{R}^\mathscr{C}_{\geq 0}$ such that $I_a \circ K(x) = A_k \circ \Psi_k(x)$ for all $x \in \mathbb{R}^\mathscr{S}_{\geq 0}$. 

\begin{remark}
In the set of power law kinetics, the complex-factorizable kinetic systems are precisely the PL-RDK systems.
\end{remark}

In some sections of this paper, kinetic orders are encoded using $T$-matrix and augmented $T$-matrix, which were introduced by Talabis et al. \cite{TALABIS2017}. These matrices are derived from the $m \times n$ matrix $\widetilde{Y}$ defined by M\"{u}ller and Regensburger in \cite{MURE2012}. In this matrix,  $( \widetilde{Y})_{ij} = F_{ki}$ if $j$ is a reactant complex of reaction $k$ and $( \widetilde{Y})_{ij} = 0$, otherwise. 

\begin{definition}
The $m \times n_r$ $\bm{T}$\textbf{-matrix} is the truncated $\widetilde{Y}$ where the non-reactant colums are deleted and $n_r$ is the number of reactant complexes. Define the $n_r \times \ell$ matrix $L=\left[e^1, e^2,\dots, e^\ell \right]$ where $e^i$ is the characteristic vector of the set of complexes $\mathscr{C}_i$ in the linkage class $\mathscr{L}_i$. That is, for all $j \in \mathscr{C}$ and $i=1,\dots, \ell$, $e^i_j=1$ if $j \in \mathscr{C}_i$ and  $e^i_j=0$ if $j \in \mathscr{C}\backslash \mathscr{C}_i$. The \textbf{augmented $\bm{T}$-matrix} is the $(m +\ell) \times n_r$ block matrix defined as $\widehat{T}=\left[ \begin{array}[center]{c} T \\  L^{\top} \\ \end{array} \right].$
\end{definition}

\begin{remark}
In \cite{TALABIS2017}, Talabis et al. defined a subclass of PL-RDK systems whose augmented $T$-matrix has maximal column rank. They called such system as \textbf{$\bm{\widehat{T}}$-rank maximal} (or of type \textbf{PL-TIK}).
\end{remark}

\subsection{Decomposition theory}
Decomposition theory was initiated by M. Feinberg in his 1987 review paper \cite{FEIN1987}. He introduced the general concept of a network decomposition of a CRN as a union of subnetworks whose reaction sets form a partition of the network’s set of reactions. He also introduced the so-called \textit{independent decomposition} of chemical reaction networks.

\begin{definition}
A decomposition of a CRN $\mathscr{N}$ into $k$ subnetworks of the form $\mathscr{N}=\mathscr{N}_1 \cup \cdots \cup \mathscr{N}_k$ is \textbf{independent} if its stoichiometric subspace is equal to the direct sum of the stoichiometric subspaces of its subnetworks, i.e., $\mathcal{S}=\mathcal{S}_1 \oplus \cdots \oplus \mathcal{S}_k$.
\end{definition}

For an independent decomposition, Feinberg concluded that any positive equilibrium of the “parent network” is also a positive equilibrium of each subnetwork.

\begin{theorem}[Rem. 5.4, \cite{FEIN1987}]  \label{feinberg theorem}
Let $(\mathscr{N},K)$ be a chemical kinetic system with partition $\{\mathscr{R}_1, \dots, \mathscr{R}_k \}$. If $\mathscr{N}=\mathscr{N}_1 \cup \cdots \cup\mathscr{N}_k$ is the network decomposition generated by the partition  and $E_+(\mathscr{N}_i,K_i)= \{ x \in \mathbb{R}^\mathscr{S}_{>0} \mid N_i K_i(x) = 0, i=1,\dots,k \}$, then 
$ \bigcap_{i=1}^kE_+ (\mathscr{N}_i, K_i)\subseteq E_+ (\mathscr{N}, K)$. If the network decomposition is independent, then equality holds.
\end{theorem}

Farinas et al. \cite{FAML2021} introduced the concept of incidence independent decomposition that is patterned after independent decomposition but considers the images of the incidence maps instead of the stoichiometric subspaces.

\begin{definition}
A decomposition of a CRN $\mathscr{N}$ into $k$ subnetworks of the form $\mathscr{N}=\mathscr{N}_1 \cup \cdots \cup \mathscr{N}_k$ is \textbf{incidence independent} if the image of the incidence map of $\mathscr{N}$  is equal to the direct sum of the images of the incidence maps of its subnetworks, i.e., $\text{Im }I_a=\text{Im }I_{a,1} \oplus \cdots \oplus \text{Im }I_{a,k}$.
\end{definition}

The following result shows the relationship between the set of incidence independent decompositions and the set of complex balanced equilibria of any kinetic system.  It is the precise analogue of Theorem \ref{feinberg theorem} for incidence independent decomposition.

\begin{theorem}[Theorem 4, \cite{FAML2020}]
\label{th:Z}
Let $(\mathscr{N},K)$ be a chemical kinetic system with decomposition  $\mathscr{N}=\mathscr{N}_1 \cup \cdots \cup\mathscr{N}_k$ and $Z_+(\mathscr{N}_i,K_i)= \{ x \in \mathbb{R}^\mathscr{S}_{>0} \mid I_{a,i} K_i(x) = 0, i=1,\dots,k \}$, then $ \bigcap_{i=1}^k Z_+ (\mathscr{N}_i, K_i)\subseteq Z_+ (\mathscr{N}, K)$. If the network decomposition is incidence independent, then equality holds and $ Z_+ (\mathscr{N}, K) \neq \emptyset \Rightarrow Z_+ (\mathscr{N}_i, K_i) \neq \emptyset$ for each $i=1,\dots,k$.
\end{theorem}

\section{A review of the kinetic representations of two pre-industrial carbon cycle models}\label{sec3}

In this Section, we review the two models of pre-industrial carbon cycle whose power law kinetic representations form the basis of our comparative analysis. 

The first model is a pre-industrial reduction of the simple mass balance model of the Earth system of R. Schmitz \cite{SCHM2002}. On the other hand, the second model is based on the analysis of the Earth’s carbon cycle in the pre-industrial state done by Anderies et al. \cite{AND2013}. For both systems, the transfer rate functions (that are not power law functions) were approximated using a standard method in Biochemical Systems Theory \cite{VOIT2006,VOIT2013,VOIT2000} to derive ODE systems with purely power law terms. For each ODE system, a dynamically equivalent chemical kinetic system (of PLK type) is constructed using the procedure developed by Arceo et. al \cite{ARCEO2015}. For detailed computations, the reader may refer to  the work of Fortun et al. in \cite{FLRM2019} and \cite{DOA2018}.

In the pre-industrial state of the model of Schmitz, six state variables $M_1,\cdots, M_6$ representing the major carbon pools are considered. Figure \ref{schmitz1}(a) provides a schematic diagram of the model. Its dynamically equivalent PLK system is a PL-NDK system with 6 species, 6 monomolecular complexes (i.e., complexes with only one species with stoichiometric coefficient of 1), and 13 reactions. Figure \ref{schmitz1}(b) presents the underlying CRN of the system, \ref{schmitz1}(c) the kinetic orders of the rate functions of each reaction, and \ref{schmitz1}(d) the CRN's relevant network numbers.

\begin{figure}[t]%
\centering
\includegraphics[width=1\textwidth]{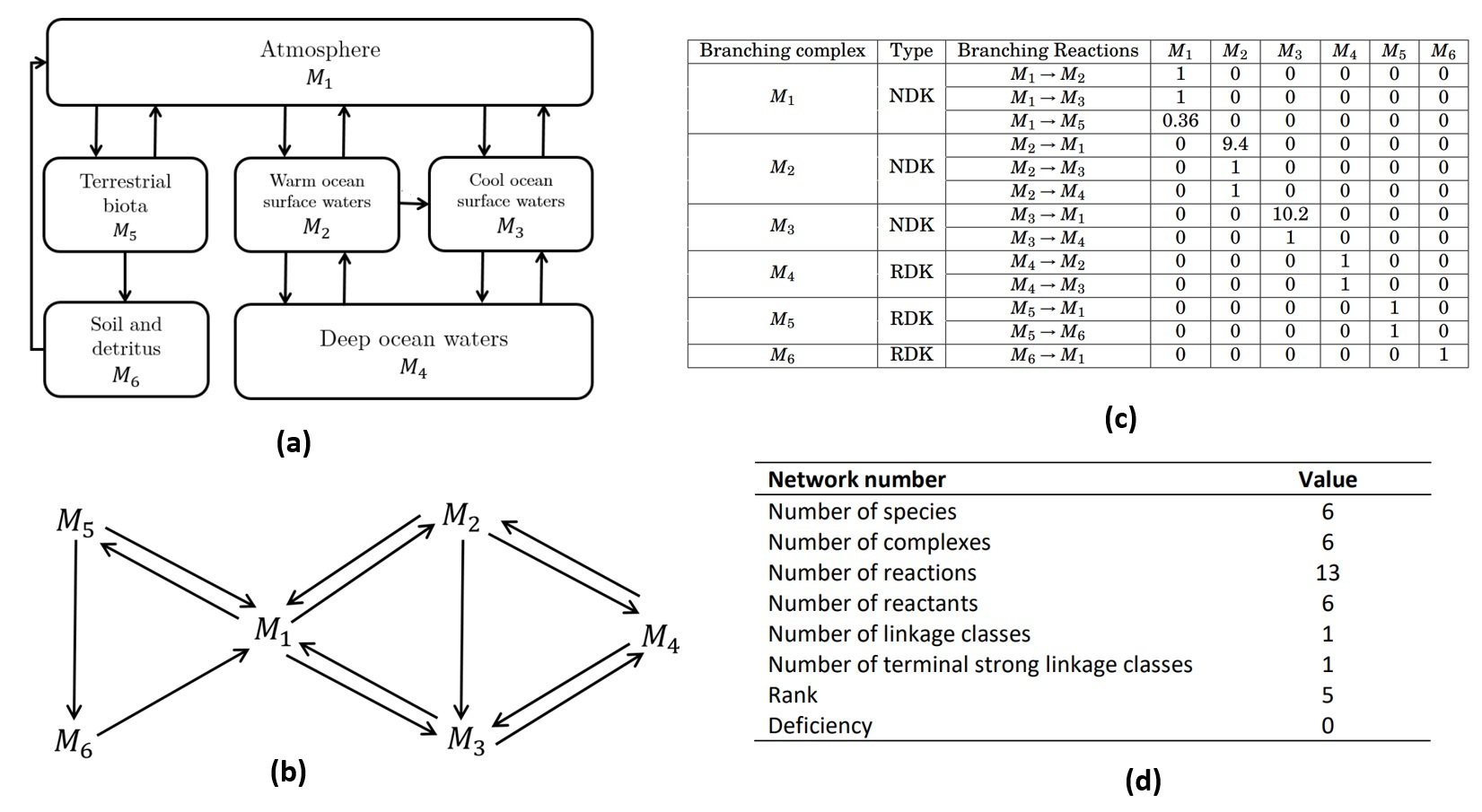}
\caption{The pre-industrial carbon cycle model based from Schmitz \cite{SCHM2002}: (a) its schematic diagram, (b) its CRN representation, (c) its kinetic order matrix, and (d) its CRN numbers.}\label{schmitz1}
\end{figure}

For the model of Anderies et al., the system involves only three major carbon pools. In Figure \ref{anderies1}(a), the boxes represent these pools. The CRN representation of its dynamically equivalent PLK system is in Figure \ref{anderies1}(b) with some network properties listed in Figure \ref{anderies1}(d). The power law dynamics of the system is encoded in the kinetic order matrix in Figure \ref{anderies1}(c). There were two kinetic representations computed in \cite{DOA2018,FLRM2021}. The difference in the computed kinetic order approximations is due to the variation of a single parameter in the original model. This parameter (denoted by $\alpha$ in the original paper) represents the human terrestrial carbon off-take rate, which accounts for that reduction of the carbon capture capacity of terrestrial systems (e.g., farming, forest clearing and burning). In \cite{DOA2018}, the approximation assumed that $\alpha =0.3$ (i.e., there are human activities that hinder carbon sequestration of land), which is a similar value used by Anderies et al. in their analysis. The power-law approximation led to the following kinetic orders: $p_1=-1.894, p_2=-0.271, q_1=0.426$, and $q_2=0.439$. On the other hand, the approximation done in \cite{FLRM2021} assumed the absence of such human activities or $\alpha = 0$. The resulting kinetic orders were $p_1=p_2=-68, q_1=0.580$, and $q_2= 0.911$.

\begin{figure}[t]%
\centering
\includegraphics[width=1\textwidth]{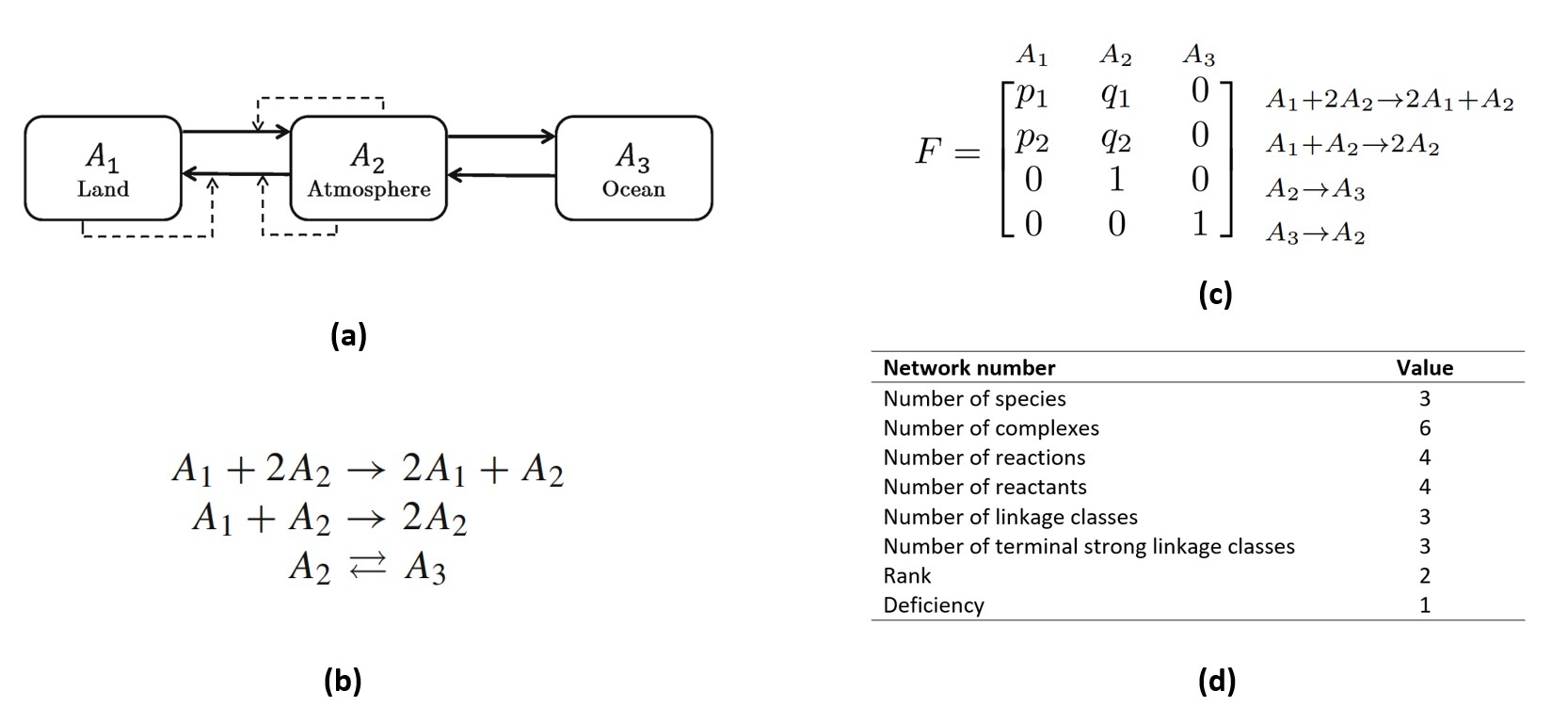}
\caption{The pre-industrial carbon cycle model based from Anderies et al. \cite{AND2013}: (a) its schematic diagram (the solid arrows connecting the pools indicate the transfer of carbon, while the dashed arrows show modulating influence), (b) its CRN representation, (c) its kinetic order matrix, and (d) its CRN numbers.}\label{anderies1}
\end{figure}

Aside from the apparent contrast in the number of major carbon pools considered in the two models, we take note of the significant differences in the assumptions of these models. In the biochemical map of the first system, all carbon fluxes are not influenced by the other components in the system (Figure \ref{schmitz1}(a)). Moreover, the transfer rate function in each flux is described basically by the product of a mass transfer coefficient and a concentration function of the “reactant” pool. Meanwhile, the second model considers biogeochemical feedback or modulating influences in the land-atmosphere interaction. In Figure \ref {anderies1}(a), these feedback or regulatory effects on carbon transfer are depicted by the dashed arrows. One reason for this feedback mechanism is that, unlike the first model, temperature is not fixed in the second system.  

This non-isothermal premise is embedded in the definition of the feedback functions respiration (land to atmosphere) and photosynthesis (atmosphere to land), which are expressed as functions of temperature. In turn, the temperature depends linearly on atmospheric carbon or $A_2$. Hence, these feedback mechanisms are composite functions of $A_2$. Moreover, the modulating arrow in the land-to-atmosphere carbon transfer is due to a logistic function dependent on land ($A_1$) which accounts for competition for space, sunlight, water, or nutrients.

In Schmitz's model, photosynthesis and respiration are not considered as functions but they are basically embedded in the estimates for  carbon transfer rate constants in the land-atmosphere interaction.

Furthermore, aside from the temperature, the human terrestrial off-take quantity (as described earlier) is not constant in the Anderies et al. system. The model of Schmitz accounts for a similar quantity, which is called ``anthropogenic disturbances” by the author, but is held fixed in the analysis.

\section{Reaction network analysis of Schmitz systems}\label{sec4}

In this Section, we first derive new results on properties of the Schmitz model (Sections \ref{sec4.1} – \ref{sec4.5}). We then combine them with known properties from \cite{FLRM2019,HEME2022b,LLMM2022,NEML2019} to provide an overview in Table \ref{table:schmitz} (Section \ref{sec4.6}). Throughout the section, $(\mathscr{N}, K)$ denotes a kinetic representation described in Section \ref{sec3} and is referred to as the \textbf{Schmitz system}. We also use the term \textbf{Schmitz network} for $\mathscr{N}$.

\subsection{The structure of the positive equilibria set of a Schmitz system} \label{sec4.1}
To date, the only known property of $E_+(\mathscr{N}, K)$ is that it is non-empty and a subset consists of elements ``lifted" from a subsystem $(\mathscr{N}',K')$ shown in \cite{FLRM2019}. In this section, we use the concordance of $\mathscr{N}$  and the weak monotonicity of its kinetics to show that it has a unique positive equilibrium in each stoichiometric class. 

\subsubsection{Properties of concordant networks and weakly monotonic kinetics}

Concordant networks were introduced by G. Shinar and M. Feinberg in 2012 \cite{SHFE2012} as an abstraction of continuous flow stirred tank reactors (CFSTRs), a widely used model in chemical engineering. In their view, concordance indicates ``\ldots architectures that by their very nature, enforce duller, more restrictive behavior despite what might be great intricacy in the interplay of many species, even independently of values that kinetic parameters might take". Concordance can hence be seen as a new type of system stability. To precisely define concordance, consider the linear map $L:\mathbb{R}^\mathscr{R} \rightarrow \mathcal{S}$ defined by
$$L (\alpha) = \sum_{y \rightarrow y'} \alpha_{y \rightarrow y'} (y'-y).$$

%The formal definition of concordance is as follows:
\begin{definition}
A reaction network $\mathscr{N}$ is \textbf{concordant} is there do not exist an $\alpha \in \text{Ker } (L)$ and a nonzero $\sigma \in \mathcal{S}$ having the following properties:
\begin{enumerate}[label=(\roman*)]
    \item For each $y \rightarrow y' \in \mathscr{R}$ such that $\alpha_{ y \rightarrow y' } \neq 0$, $\text{supp}(y)$ contains a species $S$ for which $\text{sgn} (\sigma_S) = \text{sgn} (\alpha_{ y \rightarrow y' })$, where $\sigma_S$ denotes the term in $\sigma$ involving the species $S$ and $\text{sgn}(\cdot)$ is the signum function.
    \item For each $y \rightarrow y' \in \mathscr{R}$ such that $\alpha_{y \rightarrow y' } =0$, either $\sigma_S=0$ for all $S \in \text{supp}(y)$, or else $\text{supp}(y)$ contains species $S$ and $S'$ for which $\text{sgn} (\sigma_S) = -\text{sgn}(\sigma_{S'})$, but not zero.
\end{enumerate}
A network that is not concordant is \textbf{discordant}.
\end{definition}

Concordance is closely related to two classes of kinetics on a network: injective  and weakly monotonic kinetics. We recall these notions from \cite{SHFE2012} here.

\begin{definition}
A kinetic system $(\mathscr{N}, K)$ is \textbf{injective} if, for each pair of distinct stoichiometrically compatible vectors $x^*, x^{**} \in \mathbb{R}_{\geq 0}^\mathscr{S}$, at least one of which is positive,
$$ \sum_{y \rightarrow y'} K_{y \rightarrow y'} (x^{**}) (y' -y) \neq \sum_{y \rightarrow y'} K_{y \rightarrow y'} (x^{*}) (y' -y).$$
\end{definition}
\noindent Note that an injective kinetic system is necessarily a monostationary system. Moreover, an injective kinetic system cannot admit two distinct stoichiometrically compatible equilibria, at least one of which is positive.

\begin{definition}
A kinetics $K$ for a reaction network $\mathscr{N}$ is \textbf{weakly monotonic} if, for each pair of vectors $x^*, x^{**} \in \mathbb{R}_{\geq 0}^\mathscr{S}$, the following implications hold for each reaction $y \rightarrow y' \in \mathscr{R}$ such that $\text{supp} (y) \subset \text{supp}(x^*)$ and $\text{supp} (y) \subset \text{supp}(x^{**})$: 
\begin{enumerate}[label=(\roman*)]
    \item $K_{y \rightarrow y'} (x^{**}) > K_{y \rightarrow y'} (x^{*}) \implies$ there is a species $S \in \text{supp} (y)$ with $x^{**}_S > x^{*}_S$. 
    \item $K_{y \rightarrow y'} (x^{**}) = K_{y \rightarrow y'} (x^{*}) \implies$ $x^{**}_S  = x^{*}_S$ for all $S \in \text{supp} (y)$ or else there are species $S, S' \in \text{supp} (y)$ with $x^{**}_S  > x^{*}_S$ and $x^{**}_{S'}  < x^{*}_{S'}$.
\end{enumerate}
\end{definition}

\begin{remark}\label{rem:NIK}
Examples of weakly monotonic kinetic systems are mass action systems and a class of power law systems where all kinetic orders are non-negative called non-inhibitory kinetics or \textbf{PL-NIK systems} in \cite{ARCEO2015}. 
\end{remark}

The following two propositions present the close relationship between concordant networks, injective and weakly monotonic kinetics: 

\begin{proposition} (Proposition 4.8 of \cite{SHFE2012})
A weakly monotonic kinetic system $(\mathscr{N},K)$ is injective whenever its underlying reaction network is concordant. In particular, if the underlying reaction network is concordant, then the kinetic system cannot admit two distinct stoichiometrically compatible equilibria, at least one of which is positive.
\end{proposition}

\begin{theorem}\label{theorem:shfe2012} (Theorem 4.11 of \cite{SHFE2012})
A reaction network has injectivity in all weakly monotonic kinetic systems derived from it if and only if the network is concordant.
\end{theorem}

\noindent The previous statement shows that concordant networks are in fact characterized by any weakly monotonic kinetics on it being necessarily injective.

\subsubsection{Any Schmitz system is a Birch system}
\begin{figure}[t]%
\centering
\includegraphics[width=0.6\textwidth]{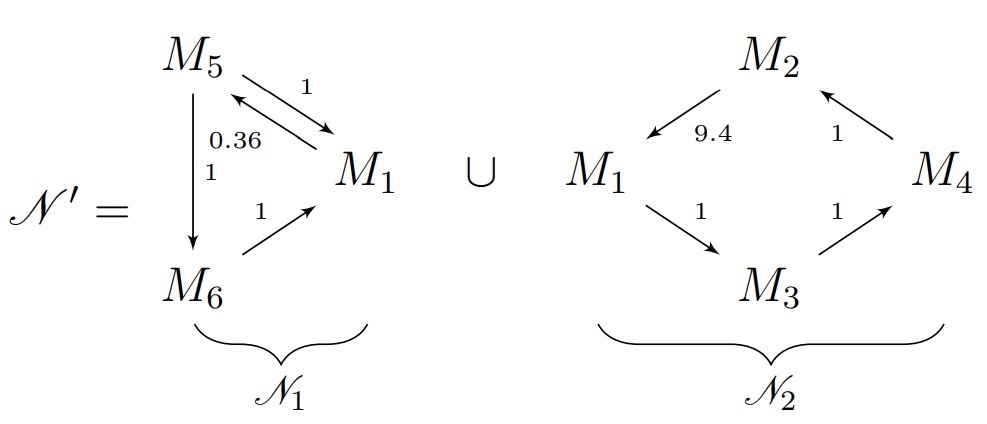}
\caption{The subnetwork $\mathscr{N}'= \mathscr{N}_1 \cup \mathscr{N}_2$ of the Schmitz system investigated in \cite{FLRM2019}}\label{fig3}
\end{figure}

Fortun et al. \cite{FLRM2019} studied the equilibria of the Schmitz's PL-NDK system $(\mathscr{N}, K)$ by identifying a subnetwork $(\mathscr{N}', K')$ with an independent decomposition into two PL-RDK systems and using a result of Joshi and Shiu \cite{JOSH2012} to ``lift" the equilibria of the subnetwork to the whole system. Figure \ref{fig3} shows the subnetwork $\mathscr{N}'= \mathscr{N}_1 \cup \mathscr{N}_2$. Here, we show that both the Schmitz network and its subnetwork form in fact Birch systems, which are defined as follows:

\begin{definition}
A kinetic system is a \textbf{Birch system} if it has a unique positive equilibrium in every stoichiometric class and this equilibrium is complex balanced.
\end{definition}

\noindent The basis for the claimed property of the Schmitz system and its subnetwork is the Shinar-Feinberg Positive Equilibria Theorem:

\begin{theorem} (Theorem 6.8 of \cite{SHFE2012})
If $K$ is a continuous kinetics for a conservative reaction network $\mathscr{N}$, then the kinetic system $(\mathscr{N},K)$ has an equilibrium within each stoichiometric compatibility class. If the network is weakly reversible and concordant, then within each nontrivial stoichiometric compatibility class there is a positive equilibrium. If, in addition, the kinetics is weakly monotonic, then that positive equilibrium is the only equilibrium in the stoichiometric compatibility class containing it.
\end{theorem}

\begin{proposition}
Any Schmitz system $(\mathscr{N}, K)$ and its subsystem $(\mathscr{N}', K')$ are Birch systems.
\end{proposition}

\begin{proof}
The systems are weakly reversible and via the CRNToolbox test \cite{CRNTool}, concordant. Their stoichiometric subspaces are known to coincide, and the vector $(1,1,1,1,1,1)$ is easily verified to be in the orthogonal complement, implying that they are both conservative. Since the kinetics is PL-NIK on both (by Remark \ref{rem:NIK}), the Shinar-Feinberg theorem cited above implies that the systems have a unique positive equilibrium in each stoichiometric class. Since they both have zero deficiency, it follows from Feinberg's classical result that each positive equilibrium is complex balanced (see Remark \ref{rem:CB}).
\end{proof}

\begin{corollary}
There is a bijection $\omega: E_+(\mathscr{N}, K) \rightarrow E_+(\mathscr{N}', K')$.
\end{corollary}
\begin{proof}
Each equilibria set has a unique member in a stoichiometric class, and the underlying networks have the same stoichiometric classes.
\end{proof}

%\begin{remark} Joshi and Shiu \cite{JOSH2012} point out that since they employ homotopy theory arguments for the equilibrium ``lifting", the equilibrium on the subnetwork is not necessarily the same as the ``lifted" one on the network. In Section \ref{}, we show that for the Schmitz network and its subnetwork, the bijection is most likely not the identity map. \end{remark}

\subsubsection{The finest independent and incidence independent decompositions of the Schmitz network} \label{sec:finest}

A criterion for the existence of a non-trivial independent decomposition and an algorithm to determine the finest such decomposition were presented by Hernandez and de la Cruz \cite{HECR2021}. The algorithm is easily adapted to find the finest incidence-independent decomposition of a network (\cite{HEAM2022}). Such decompositions have been found to be very useful in reaction network analysis of large networks, e.g. metabolic insulin signaling (\cite{LUML2022}). We compute them to highlight a difference between the two carbon cycle models.

The finest independent decomposition of the Schmitz network is given by the subnetworks $\mathscr{N}_1$  and $\mathscr{N}_2 \cup \mathscr{N}_3$, where 
\begin{align*}
    \mathscr{R}_1 &= \{ M_1 \leftrightarrows M5, M_5 \rightarrow M_6, M_6 \rightarrow M_1 \},  \\
    \mathscr{R}_2 &= \{ M_2 \rightarrow M_1, M_1 \rightarrow M_3, M_3 \rightarrow M_4, M_4 \rightarrow M_2 \}, \\
    \mathscr{R}_3 &= \{M_1 \rightarrow M_2, M_2 \rightarrow M_3, M_2 \rightarrow M_4, M_4 \rightarrow M_3, M_3 \rightarrow M_1 \}.
\end{align*}
It is also the finest incidence-independent decomposition because the network is monomolecular, implying the equivalence of independence and incidence independence.

\subsection{ACR and equilibria co-multiplicity analysis in Schmitz systems} \label{sec4.2}

In this Section, we show that the Schmitz kinetic representation does not possess ACR in any of its six species. This property is derived by exhibiting distinct positive equilibria whose coordinates differ in all species. Furthermore, we study the relationship of ACR to the novel concepts of equilibria co-monostationarity and co-multistationarity.

\subsubsection{Absence of ACR in Schmitz systems}

Absolute concentration robustness or ACR refers to a condition in which the concentration of a species in a network attains the same value in every positive steady-state set by parameters and does not depend on initial conditions. This concept was introduced by Shinar and Feinberg in their well-cited paper \cite{SHFE2010} published in Science in 2010.  Notably, they presented sufficient structure-based conditions for a mass action system to display ACR on a particular species. This result was extended to power law kinetic systems of low deficiency \cite{FLRM2021,FOME2021}, subsets of poly-PL kinetic systems \cite{LLMM2022}, and Hill-type kinetic systems \cite{HEME2021}. For larger systems and those with higher deficiency, independent decomposition helps identify ACR \cite{FMF2021}. In \cite{HEME2022}, a general approach extended the species hyperplane approach introduced in \cite{LLMM2022}.

The basis of our ACR analysis is the following Proposition:

\begin{proposition}\label{prop:schmitzEquilibria}
The set of positive equilibria of any Schmitz system can be parametrized by $M_2$ as follows:
\begin{align*} 
M_2 &= M_2 \\ 
M_3 &=\text{Root of } \lbrace k_{31}k_{12} \left( k_{42} + k_{43} \right)M_3^{10.2} + k_{42}k_{34} (k_{12} +k_{31}) M_3 \\
& - \left( \left[ k_{23} (k_{42} +k_{43}) +k_{24}k_{43} \right](k_{12}+k_{13})M_2 + k_{21} k_{13} (k_{42}+k_{43})M_2^{9.4} \right) = 0 \rbrace \\
M_1 &=\dfrac{k_{21}M_2^{9.4}+k_{31}M_3^{10.2}}{k_{12}+k_{13}} \\
M_4 &=\dfrac{k_{24}M_2+k_{34}M_3}{k_{42}+k_{43}} \\
M_5 &=\dfrac{k_{15}}{k_{51}+k_{56}} \left( \dfrac{k_{21}M_2^{9.4}+k_{31}M_3^{10.2}}{k_{12}+k_{13}} \right)^{0.36} \\
M_6 &=\dfrac{k_{56}k_{15}}{k_{61}(k_{51}+k_{56})} \left( \dfrac{k_{21}M_2^{9.4}+k_{31}M_3^{10.2}}{k_{12}+k_{13}} \right)^{0.36} \\
\end{align*}
\end{proposition}
\begin{proof}
Consider the finest independent decomposition $\mathscr{N}_1 $ and $\mathscr{N}_2 \cup \mathscr{N}_3$ of the Schmitz system described in Section \ref{sec:finest}, where $\mathscr{R}_1 = \{ M_1 \leftrightarrows M5, M_5 \rightarrow M_6, M_6 \rightarrow M_1 \}$ and $\mathscr{R}_2 \cup \mathscr{R}_3=\{ M_1 \leftrightarrows M2, M_1 \leftrightarrows M3, M_2 \rightarrow M3, M_2 \leftrightarrows M4, M_3 \leftrightarrows M4 \} $. Hence, the equilibria of the system can be obtained by taking the intersection of the equilibria set of $\mathscr{N}_1$ and the equilibria set of $\mathscr{N}_2 \cup \mathscr{N}_3$.

The ODE system of $\mathscr{N}_1$ is given by
\begin{align}
\dot{M}_1 &= k_{51}M_5 + k_{61} M_6 - k_{15}M_1^{0.36} \label{eq1}\\
\dot{M}_5 &= k_{15}M_1^{0.36} - k_{51}M_5 -k_{56}M_5 \label{eq2} \\
\dot{M}_6 &= k_{56}M_5 - k_{61} M_6 \label{eq3}
\end{align}
Setting the above equations to 0, we obtain
\begin{equation}
M_5 = \dfrac{k_{15}}{k_{51}+k_{56}} M_1^{0.36} \quad \text{and} \quad 
M_6 = \dfrac{k_{56}}{k_{61}}M_5= \dfrac{k_{56}}{k_{61}}\cdot \dfrac{k_{15}}{k_{51}+k_{56}} M_1^{0.36}.
\end{equation}
Hence the equilibria set $E_1$ of $\mathscr{N}_1$ is comprised of vectors that can be parametrized by $M_1$. That is,
$$
E_1 (\mathscr{N}_1,K) = \left\{ \left( M_1, \dfrac{k_{15}}{k_{51}+k_{56}} M_1^{0.36},  \dfrac{k_{56}k_{15}}{k_{61}(k_{51}+k_{56})} M_1^{0.36} \right) \middle\vert M_1 \in \mathbb{R}_{>0}  \right\}.
$$
On the other hand, the ODE system of $\mathscr{N}_2 \cup \mathscr{N}_3$ is given by
\begin{align}
\dot{M}_1 &= k_{21} M_2^{9.4} +k_{31}M_3^{10.2} -k_{12}M_1 -k_{13}M_1 \label{eq5}\\
\dot{M}_2 &= k_{12}M_1 + k_{42}M_4 -k_{23}M_2 - k_{24}M_2 - k_{21}M_2^{9.4}  \label{eq6}\\
\dot{M}_3 &= k_{13}M_1 + k_{23}M_2 + k_{43}M_4 - k_{34}M_3 - k_{31}M_3^{10.2} \label{eq7}\\
\dot{M}_4 &= k_{24}M_2 + k_{34}M_3 -k_{42}M_4 - k_{43}M_4 \label{eq8}
\end{align}
Set each equation to 0.
From Equations (\ref{eq5}), (\ref{eq6}), and (\ref{eq8}), we respectively get
\begin{align}
M_1 &= \dfrac{k_{21}M_2^{9.4}+k_{31}M_3^{10.2}}{k_{12}+k_{13}} \label{eq9}\\
M_1 &= \dfrac{(k_{23}+k_{24}) M_2+k_{21} M_2^{9.4}-k_{42} M_4}{k_{12}} \label{eq10}\\
M_4 &=\dfrac{k_{24}M_2 + k_{34}M_3}{k_{42}+k_{43}} \label{eq11}
\end{align}
Combining Eqs. (\ref{eq10}) and (\ref{eq11}), we get
\begin{equation}
M_1 = \dfrac{(k_{23}+k_{24})M_2+k_{21}M_2^{9.4}}{k_{12}} - \dfrac{k_{42}}{k_{12}} \left(\dfrac{k_{24}M_2 + k_{34}M_3}{k_{42}+k_{43}} \right). \label{eq12}
\end{equation}
Comparing Eqs. (\ref{eq9}) and (\ref{eq12}), we obtain
\begin{equation}
\dfrac{k_{21}M_2^{9.4}+k_{31}M_3^{10.2}}{k_{12}+k_{13}} = \dfrac{(k_{23}+k_{24})M_2+k_{21}M_2^{9.4}}{k_{12}} - \dfrac{k_{42}}{k_{12}} \left(\dfrac{k_{24}M_2 + k_{34}M_3}{k_{42}+k_{43}} \right).
\end{equation}
To express $M_3$ in terms of $M_2$, we can manipulate the previous equation to get
\begin{align*}
& k_{31} k_{12} (k_{42}+k_{43} ) M_3^{10.2}+k_{42} k_{34} (k_{12}+k_{31} ) M_3 = \\ & \left[k_{23}(k_{42}+k_{43} )+k_{24} k_{43} \right](k_{12}+k_{13} ) M_2+k_{21} k_{13} (k_{42}+k_{43} ) M_2^{9.4}
\end{align*}
Hence, each of the components of an equilibrium of any Schmitz system is expressible as $M_2$.
\end{proof}

\begin{corollary}
There are no ACR species in a Schmitz system.
\end{corollary}
\begin{proof}
 By inspection, two different values for $M_2$ will yield different equilibria sets varying in respective components. Hence, the system will not yield ACR in any species.
\end{proof}

\subsubsection{ACR analysis of conservative closed reaction networks of maximal}\label{sec4.2.2}

A reaction network is called \textbf{open} if its stoichiometric subspace is the entire species space, i.e., $s=m$; otherwise, it called \textbf{closed}. Hence, a closed network has maximal rank $m-$1 (and minimal rank = 1). We have observed an interesting interplay of structural and kinetic properties on such network. In this Section, for conservative closed networks of maximal rank, we show that new ``dual" properties and the occurrence of absolute concentration robustness (ACR) in species are closely interrelated. %In Section \ref{}, we establish connections between conservativeness and a new ``dual" property with equilibria multiplicity for minimal rank networks.  

We introduce new ``dual" concepts for mono- and multistationarity:

\begin{definition}
A kinetic system $(\mathscr{N},K)$ (with at least two distinct equilibria) is \textbf{co-monostationary} if each \textbf{co-stoichiometric class} $x + \mathcal{S}^\perp $ contains at most one positive equilibrium. Similarly, a system is called \textbf{co-multistationary} if there is a co-stoichiometric class containing two distinct equilibria.
\end{definition}
\noindent The term is derived from the definition of a monostationary (multistationary) network where each stoichiometric class $x + \mathcal{S}$ contains at most (more than) one positive equilibrium. 

\begin{proposition}
Let $\Delta E_+$ be the set of equilibria differences of $(\mathscr{N},K)$. Then $(\mathscr{N},K)$ is co-mononostationary if and only if  $\Delta E_+ \cap \mathcal{S}^\perp = \{0\}$.
\end{proposition}
\noindent The proof is straightforward.

\begin{proposition} 
Any Schmitz system is co-monostationary.
\end{proposition}
\begin{proof}
 Lao et al. \cite{LLMM2022} showed that $(1,1,1,1,1,1)$ is a basis for $\mathcal{S}^\perp$ where $\mathcal{S}$ is the stoichiometric subspace of any Schmitz system, showing that it is a conservative, closed kinetic system of maximal rank. If the difference of two distinct equilibria were in $\mathcal{S}^\perp$, it would be a positive multiple of the basis vector, which is not possible in view of the parametrization in Proposition \ref{prop:schmitzEquilibria}.
\end{proof}

We establish the connection beween the occurence of ACR and co-monostationarity for conservative closed network with maximal rank:

\begin{proposition}\label{prop:ACRthenco-mono}
Let $\mathscr{N}$ be a conservative closed network of maximal rank. If $(\mathscr{N},K)$ has an ACR species, then $(\mathscr{N},K)$ is co-monostationary.
\end{proposition}
\begin{proof}
 Suppose, on the contrary, there are equilibria in a co-stoichiometric class.  Then their difference is in $\mathcal{S}^\perp$, but has a zero in the coordinate of the ACR species, a contradiction.
\end{proof}

\begin{remark}
While co-multistationarity is a necessary condition for ACR in conservative, closed kinetic systems with maximal rank, any Schmitz system shows that it is not sufficient.
\end{remark}

\subsection{Kinetic/Stoichiometric Subspace Coincidence (KSSC) of a Schmitz system} \label{sec4.3}

We use a result of Nazareno et al. \cite{NEML2019} to show that the Schmitz system has the KSSC property.

\subsubsection{A brief review of KSSC results}

The coincidence of the kinetic and stoichiometric subspaces of a kinetic system is a necessary condition for the existence of non-degenerate (and subsequently stable) equilibria. Furthermore, if two systems have the KSSC property, any dynamic equivalence between them also leaves the stoichiometric subspace invariant. For mass action systems, M. Feinberg \cite{FEIN2019} observes an extreme ``lack of robustness" of system properties in systems without KSSC.

M. Feinberg and F. Horn \cite{FEHO1977} derived the “classical” KSSC Theorem for mass action systems in 1977:

\begin{theorem} (\cite{FEHO1977})
Let $\mathcal{K}$ be the kinetic subspace of a mass action system.
\begin{enumerate}[label=(\roman*)]
    \item If $t - \ell =0$, then $\mathcal{K}=\mathcal{S}$.
    \item If $t - \ell > \delta$, then $\mathcal{K} \neq \mathcal{S}$.
    \item If $0<t - \ell \leq \delta$, $\mathcal{K}=\mathcal{S}$ or $\mathcal{K} \neq \mathcal{S}$ depending on the rate constants.
\end{enumerate}
\end{theorem}

A striking feature of this result is that a single network property, the value of the difference ``$t - \ell$", determines the coincidence or non-coincidence of the subspaces. Arceo et al. \cite{AJLM2017} extended the result to the factor span surjective (FSK) subset of complex factorizable kinetic (CFK) systems forty years later. 

\begin{theorem} \label{theore:ajlm2017} (Theorem 3 of \cite{AJLM2017})
For a complex factorizable system on a network $\mathscr{N}$,
\begin{enumerate}
    \item[(i)] if $t-\ell > \delta$, then $\mathcal{K} \neq \mathcal{S}$.
    \item[(i')] if $0<t - \ell \geq \delta$, and a positive steady state exists, then $\mathcal{K} \neq \mathcal{S}$. In fact, $\dim \mathcal{S} - \dim \mathcal{K} \geq t-\ell - \delta +1$. \\
    if the system is also factor space surjective and
    \item[(ii)] if $t - \ell =0$, then $\mathcal{K}=\mathcal{S}$.
    \item[(iii)] if $0<t - \ell \leq \delta$ and a positive steady state does not exist, then $\mathcal{K}=\mathcal{S}$ or $\mathcal{K} \neq \mathcal{S}$ depending on the rate constants.
\end{enumerate}
\end{theorem}

The kinetic conditions CFK and FSK were not visible in the result of Feinberg and Horn because all mass action systems possessed them. A disadvantage of the FSK concept is that for many kinetic systems, e.g. for power law systems, FSK systems exist only for cycle terminal networks, i.e. when each complex is a reactant complex. A concept with broader scope, called \textit{interaction span surjectivity}, was introduced first for NFK systems by Nazareno et al. \cite{NEML2019} in 2019.

Interaction span surjectivity is based on the concept of CF-subsets of a kinetic system. If we represent the kinetics at a reaction $q = y \rightarrow y’$ as 
$$K_q (x) = k_q I_{K,q} (x),$$
where $k_q$ is a positive rate constant and $I_{K,q}$ is the interaction function, then we can partition the set of reactions with reactant complex $y$ into subsets with the same interaction function. Such a subset is called a CF-subset, and the CF-subsets (whose total number we denote with $\bm{N_R}$) form a partition of the set of all reactions of the kinetic system. Each CF-subset is characterized by a pair $(y, I_K)$, where $y$ is the (identical) reactant complex and $I_K$ the (identical) interaction function of all reactions in the set.

\begin{definition}
A kinetics is \textbf{interaction span surjective} if the set of interaction functions $I_K$ of its $N_R$ CF-subsets are linearly independent.
\end{definition}

Nazareno et al. derived the following KSSC Theorem for a class of non-complex factorizable systems (called NF-RIDK systems):

\begin{theorem} \label{theorem:nazareno} (Theorem 3 of \cite{NEML2019})
Consider an NF-RIDK system $(\mathscr{N},K)$. Let $n=$ number of complexes, $n_r=$ number of reactants,  $N_R$  the number of CF-subsets, $r=$ number of reactions, $r_{mcf}=$ number of reactions in a maximal CF-subnetwork, and $s=$ rank of the CRN. 
\begin{enumerate}[label=(\roman*)]
    \item If $N_R < s$, then $\mathcal{K}=\mathcal{S}$. \\
    If the system is also interaction span surjective, then either
    \item $t-\ell =0$ and $r- r_{mcf} = N_R - n_r$ implies $\mathcal{K}=\mathcal{S}$ ; or
    \item  $t - \ell \leq \delta$ and $t=n-n_r$ implies that $\mathcal{K}=\mathcal{S}$ is rate-constant dependent.
\end{enumerate}
\end{theorem}

Recently, Arceo et al. \cite{AJLM2022} extended the interaction span surjectivity concept to CFK systems and showed that the KSSC Theorem extends from FSK to its superset of interaction span surjective kinetic (ISK) systems. In other words, one can simply replace “factor span surjectivity” with  “interaction span surjectivity” in Theorem \ref{theore:ajlm2017} above. This result provides the basis for a new result in Section \ref{sec5.1}. 

\subsubsection{The KSSC property of any Schmitz system}
To assess subspace coincidence for any Schmitz system, since PL-NDK systems are non-complex factorizable, we first attempted to apply Theorem \ref{theorem:nazareno}  (Nazareno et al.’s KSSC Theorem). Being weakly reversible, it is clearly $t$-minimal (i.e., $t-\ell=0$). The ISK property can be checked for power law systems by the following result of Arceo et al. \cite{AJLM2022}:

\begin{proposition} \label{prop:ISK} (\cite{AJLM2022})
A PLK system $(\mathscr{N}, K)$ with kinetic order matrix $F$ is ISK if and only if the rows in $F$ of any two reactions from different CF-subsets are different.
\end{proposition}

As shown in Example 2 of \cite{NEML2019}, any Schmitz system has $N_R = 9$ CF-subsets and the kinetic order rows of these are pairwise different, hence the system is an ISK system. 

The final requirement concerns the number $r_{mcf}$ of reactions in a maximal CF-subnetwork of the system. Such a subnetwork is defined by the union of all branching reactions of RDK nodes and a CF-subset of each NDK-node with the maximal number of elements. An easy computation shows that $r_{mcf} = 10$, so that $r - r_{mcf}  = 13 - 10 = N_R - n_r = 9 – 6$.  This establishes KSSC for any Schmitz system.

\subsection{A low deficiency complement of a Schmitz system} \label{sec4.4}
In this Section, we are interested in identifying, for a given power law kinetic system $(\mathscr{N}, K)$, weakly reversible PL-RDK systems $(\mathscr{N}^\#, K^\#)$ which have low deficiency (i.e., $\delta = 0$ or = 1) and identical positive equilibria sets, i.e.  $E_+ (\mathscr{N}, K)=E_+(\mathscr{N}^\#, K^\#)$. Much is known about such low deficiency systems, which could be used to understand the given system. Typical examples are linear conjugates with such properties. If $(\mathscr{N}, K)$ itself has low deficiency $\delta$ and $\delta^\# = 1- \delta$ , then we call $(\mathscr{N}^\#, K^\#)$ a \textbf{low deficiency complement (LDC)}.

\subsubsection{Invariance of network properties under linear conjugacy}
Johnston and Siegel \cite{JOSI2011} showed that two kinetic systems (with the same species space) are linearly conjugate if and only if there is a positive vector $c \in \mathbb{R}^\mathscr{S}$ (called a \textit{conjugacy vector}) such that $f_\#(x) = (\text{diag } c)f(x)$, where $f_\#, f$ are the species formation rate functions of the systems.

The following Lemma and Proposition derive the invariance of network conservativity and concordance under linear conjugacy.

\begin{lemma}\label{lem:invar}
Let $(\mathscr{N}, K)$ be a kinetic system and $(\mathscr{N^\#}, K^\#)$ a linear conjugate with the same set of species. Furthermore, assume that both systems have KSSC. Then
\begin{enumerate}[label=(\roman*)]
    \item $\mathcal{S}_\# =(\text{diag }c) \mathcal{S}$ and $\mathcal{S}_\#^\perp =(\text{diag }c^{-1}) \mathcal{S}^\perp$ where $c= [c_1, \cdots, c_m]$ is a positive conjugacy vector.
    \item The isomorphism $\text{diag } c : \mathbb{R}^\mathscr{S} \rightarrow \mathbb{R}^\mathscr{S}$ induces an isomorphism $\mathbb{R}^\mathscr{S} / \mathcal{S} \rightarrow \mathbb{R}^\mathscr{S} / \mathcal{S}_\#$, mapping stoichiometric class to stoichiometric class.
\end{enumerate}
\end{lemma}
\begin{proof}
$ $
\begin{enumerate}[label=(\roman*)]
    \item Since $(\mathscr{N}, K)$ and $(\mathscr{N^\#}, K^\#)$ has KSSC, we have $(\text{diag }c) \mathcal{S} = \text{diag } \langle \text{Im } f \rangle = \langle \text{Im } f_\# \rangle=\mathcal{S}_\#$. This implies that $\mathcal{S}_\#^\perp =(\text{diag }c^{-1}) \mathcal{S}^\perp$ because $0 = \langle x, x' \rangle =\sum x_i x'_i = \sum cx_i c^{-1}x'_i =\langle cx,c^{-1}x'\rangle$.
    \item The map $x+ \mathcal{S} \rightarrow (\text{diag }c)x + \mathcal{S}_\#$ is a well-defined linear map since $\mathcal{S}_\#=(\text{diag }c) \mathcal{S}$. Clearly, its inverse map is given by $(\text{diag }c^{-1})$, showing its bijectivity.
\end{enumerate}
\end{proof}

\begin{proposition}
Let $(\mathscr{N}, K)$ and $(\mathscr{N^\#}, K^\#)$ be as in the previous lemma. Let $r$ and $r_\#$ be the number of reactions of their respective CRNs. Then 
\begin{enumerate}[label=(\roman*)]
    \item $\mathscr{N}$ is conservative $\Rightarrow \mathscr{N^\#}$ is conservative.
    \item If $r=r_\#$, $\mathscr{N}$ is concordant $\Rightarrow \mathscr{N^\#}$ is concordant.
\end{enumerate}
\end{proposition}
\begin{proof}
$ $
\begin{enumerate}[label=(\roman*)]
    \item It follows from Lemma \ref {lem:invar} (i) that if $v$ is a positive vector in $\mathcal{S}^\perp$, then $(\text{diag }c^{-1})v$ is a positive vector in $\mathcal{S}_\#^\perp$.
    \item We use the equivalence that $\mathscr{N}$ is concordant $\Leftrightarrow$ every PL-NIK system on $\mathscr{N}$ is injective (stated in \cite{FAML2020} after M. Feinberg pointed out in an email that this is shown in the proof of Theorem \ref{theorem:shfe2012}). Suppose there is a non-injective PL-NIK kinetics $\widetilde{K} (z) =(\text{diag } \widetilde{k}) z^{\widetilde{F}}$ on $\mathscr{N^\#}$. Then, if $\widetilde{f} = N_\# \widetilde{K}$, there exist $y, y' \in \mathbb{R}^\mathscr{S}_{\geq 0}$ with $y-y' \in \mathcal{S}_\#$ and $\widetilde{f}(y)= \widetilde{f}(y')$. In view of Lemma \ref {lem:invar} (ii), we have $x,x'$ with $y=(\text{diag }c)x$, $y'=(\text{diag }c)x'$, and $x-x' \in \mathcal{S}$. Hence, we have $\widetilde{f}(y) =\widetilde{f}((\text{diag }c)x)=N_\# (\text{diag } \widetilde{k}) c^{\widetilde{F}} x^{\widetilde{F}}$. Since $ c^{\widetilde{F}}$ is an $r$-vector, we can form new rate constants $\text{diag }k= \text{diag }\widetilde{k} c^{\widetilde{F}}$. We obtain
    \begin{align*}
        0 & = \widetilde{f}(y) - \widetilde{f}(y')  \\
        & = (\text{diag } c) [N(\text{diag }k) x^{\widetilde{F}} - N(\text{diag }k) x'^{\widetilde{F}}].
    \end{align*}
    Since $(\text{diag } c)$ is an isomporphism, this implies $N(\text{diag }k) x^{\widetilde{F}} - N(\text{diag }k) x'^{\widetilde{F}}=0$. Since $x – x' \in \mathcal{S}$ and $(\text{diag }k) x^{\widetilde{F}}$ is a PL-NIK on the concordant $\mathscr{N}$, it is injective, so that $x=x'$ and consequently $y=y'$. Therefore, $\mathscr{N}^\#$ is concordant too.
\end{enumerate}
\end{proof}

\subsubsection{New properties of the LDC of a Schmitz system}
\begin{figure}[t]%
\centering
\includegraphics[width=1\textwidth]{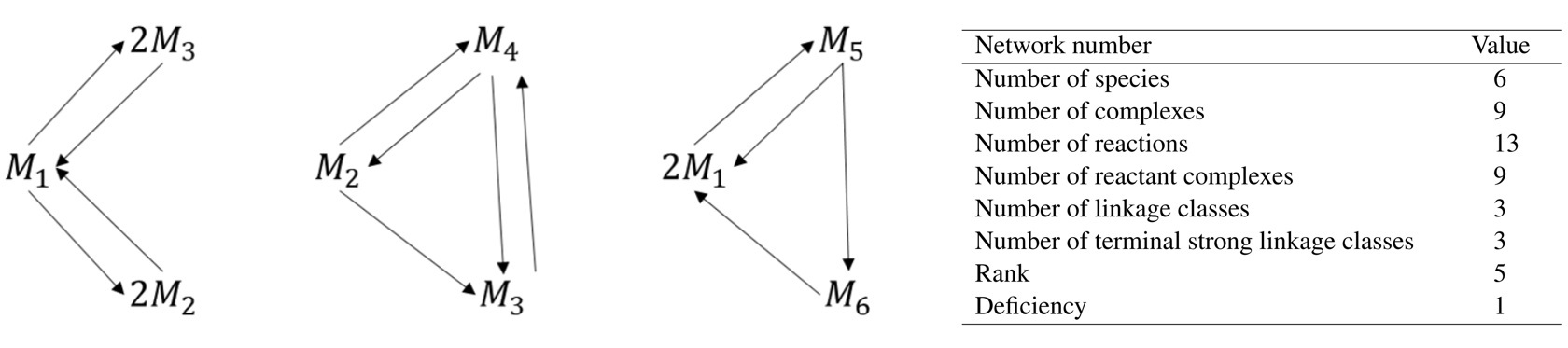}
\caption{The low deficiency complement for the Schmitz system and its network numbers.}\label{fig4}
\end{figure}
Nazareno et al. \cite{NEML2019} constructed the LDC $(\mathscr{N^\#}, K^\#)$  for the Schmitz system shown in Figure \ref{fig4}. They first transformed it to a PL-RDK system via the so-called $\text{CF-RM}_+$ method and then applied Mixed Integer Linear Programming (MILP) techniques to identify a linear conjugate weakly reversible system. Its augmented $T$-matrix below shows that it is a PL-TIK system:

\small{
$$
\widehat{T}= \begin{blockarray}{cccccccccl}
M_1 & 2M_2 & 2M_3 & M_2 & M_3 & M_4 & 2M_1 & M_5 & M_6  \\
\begin{block}{[ccccccccc]l}
1 & 0 & 0 & 0 & 0 & 0 & 0.36 & 0 & 0 & M_1 \\
0 & 9.4 & 0 & 1 & 0 & 0 & 0 & 0 & 0 & M_2  \\
0 & 0 & 10.2 & 0 & 1 & 0 & 0 & 0 & 0 & M_3  \\
0 & 0 & 0 & 0 & 0 & 1 & 0 & 0 & 0 & M_4  \\
0 & 0 & 0 & 0 & 0 & 0 & 0 & 1 & 0 & M_5  \\
0 & 0 & 0 & 0 & 0 & 0 & 0 & 0 & 1 & M_6  \\
1 & 1 & 1 & 0 & 0 & 0 & 0 & 0 & 0 & \mathscr{L}_1  \\
0 & 0 & 0 & 1 & 1 & 1 & 0 & 0 & 0 & \mathscr{L}_2  \\
0 & 0 & 0 & 0 & 0 & 0 & 1 & 1 & 1 & \mathscr{L}_3  \\
\end{block}
\end{blockarray}
$$}

We collect the new results about $(\mathscr{N}_\#,K\#)$ in the following Proposition:

\begin{proposition}
Let $(\mathscr{N^\#}, K^\#)$ be the LDC of the Schmitz system $(\mathscr{N}, K)$. Then
\begin{enumerate}[label=(\roman*)]
    \item its kinetic rank $\widetilde{s}_\#=m=6$, hence its kinetic order space $\widetilde{S}_\#=\mathbb{R}^\mathscr{S}$. 
    \item $Z_+(\mathscr{N^\#}, K^\#)$ contains a single element.
    \item $(\mathscr{N^\#}, K^\#)$ has KSSC.
    \item $(\mathscr{N^\#}, K^\#)$ is not absolutely complex balanced.
\end{enumerate}
\end{proposition}
\begin{proof}
$ $
\begin{enumerate}[label=(\roman*)]
    \item Since $(\mathscr{N^\#}, K^\#)$ is PL-TIK, it follows from Talabis et al. \cite{TAMJ2019} that it is complex balanced for all rate constants. Hence, after M\"{u}ller and Regensburger \cite{MURE2014}, its kinetic deficiency $\delta_\#=0$. Since $n_\# = N_R =9$ and $\ell_\#=3$, we get $\widetilde{s}_\#=9-3=6=m$. 
    \item The system is a complex balanced log-parametrized (CLP) system (according to the definition of Jose et al. \cite{JOMT2022} recalled in Definition \ref{def:LP}) with flux space $P_Z=\widetilde{S}_\#$. This implies, after \cite{MURE2014}, that $Z_+(\mathscr{N^\#}, K^\#)$ contains a single element.
    \item $(\mathscr{N^\#}, K^\#)$ is weakly reversible and $t$-minimal. The $T$-matrix, i.e., the submatrix of the first six rows of the matrix above, has pairwise different columns, so that the LDC is ISK (Proposition \ref{prop:ISK}). By Theorem \ref{theorem:nazareno}, it has KSSC. 
    \item After Sections \ref{sec4.1}, \ref{sec4.3}, and (iii), it follows that the LDC is weakly reversible, conservative and concordant with a weakly monotonic kinetics, so that its set of positive equilibria has a unique element in each stoichiometric class (and hence, infinite) while there is just a single complex-balanced equilibrium. Therefore, the system is not absolutely complex balanced.
\end{enumerate}
\end{proof}

\subsection{Exponential stability of equilibria} \label{sec4.5}

We recall the following definition from Meshkat et al. \cite{MESH2021}. 
\begin{definition}
A steady state $x$ is \textbf{nondegenerate} if $\text{Im } (J_x (f) \mid_\mathcal{S})=\mathcal{S}$, where $J_x (f)$ is the Jacobian matrix of $f$ at $x$. A nondegenerate steady state is \textbf{exponentially stable} (or simply, \textbf{stable}) if each of the $\dim (\mathcal{S})$ nonzero eigenvalues of $J_x (f)$ has a negative real part. If one of these eigenvalues has positive real part, then $x$ is \textbf{unstable}.
\end{definition}
It was shown in \cite{FLRM2019} that the Schmitz system has at least as many non-degenerate positive equilibria as its subnetwork $\mathscr{N}$ in Figure \ref{fig3}. The symbolic computation of the Jacobian matrix of the system at a positive steady and its eigenvalues is implemented in Maple. In particular, for the following parameters and steady state value (taken from \cite{FLRM2019,SCHM2002}):
\begin{equation*}
\begin{aligned}
k_{12} &= 0.0931 \\
k_{13} &= 0.0311 \\
k_{15} &= 10.08896 \\
k_{21} &= 58 \cdot 730^{-9.4} \\
k_{23} &= 0.0781 \\
k_{24} &= 0.0164 \\
\end{aligned}
\qquad \qquad
\begin{aligned}
k_{31} &= 18 \cdot 140^{-10.2} \\
k_{34} &= 0.714 \\
k_{42} &= 0.00189 \\
k_{43} &= 0.00114 \\
k_{51} &= 0.0862 \\
k_{56} &= 0.0862 \\
k_{61} &= 0.0333
\end{aligned}
\qquad \qquad
\begin{aligned}
M_1 &= 612 \\
M_2 &= 730 \\
M_3 &= 140.424 \\
M_4 &= 37041.164 \\
M_5 &= 579.080 \\
M_6 &= 1499,
\end{aligned}
\end{equation*}

\noindent we obtain 5 (which is equal to $\dim (\mathcal{S})$) non-zero eigenvalues have negative real parts:
$$    \lambda = -2.08129,-0.94080,-0.19917, -0.06698,- 0.00825. $$
Hence, the stability of the steady state is confirmed.

\subsection{Properties of a Schmitz kinetic system: an overview} \label{sec4.6}

Table \ref{table:schmitz} provides an overview of the salient network and kinetic properties of the Schmitz systems. In the References column, the last entry indicates the publication containing the result, the preceding ones are sources for related concepts and results.

\begin{table}[!ht]
%\small
%\def\arraystretch{1.2}
\begin{center}
\begin{minipage}{\textwidth}
\caption{Overview of network and kinetic properties of the Schmitz model.}
\label{table:schmitz}
\begin{tabular}{|c|l|c|}
\hline
\textbf{Property class} & \textbf{Schmitz system} & \textbf{References} \\ 
\hline
\multirow{3}{*}{Network}    &   connected   &   \cite{FLRM2019} \\
\hhline{~--}  & weakly reversible & \cite{FLRM2019} \\
\hhline{~--} (Digraph)& cycle terminal &    \cite{FLRM2019} \\
\hhline{~--} &  $t$-minimal &   \cite{FLRM2019} \\
\hline
\multirow{8}{*}{Network} & concordant   & \cite{SHFE2012}, this paper   \\
\hhline{~--} & conservative & this paper  \\
\hhline{~--} & monomolecular & \cite{MJS2018} \\
\hhline{~--} &  zero deficiency & \cite{FLRM2019} \\
\hhline{~--} (Stochiometry-related) &  positive dependent & \cite{FLRM2019} \\
\hhline{~--} &  stoichiometric subspace is hyperplane & \cite{FLRM2019} \\
\hhline{~--} &  (trivially) ILC & \cite{MJS2018} \\
\hhline{~--} &  finest independent decomposition =  & \multirow{2}{*}{\cite{HEAM2022}, this paper} \\
\hhline{~~~} &  finest incidence independent decomposition & \\
\hline
\multirow{16}{*}{ Kinetic system}  & PL-NDK &  \cite{FLRM2019}  \\
\hhline{~--} & PL-NIK & \cite{FLRM2019}  \\
\hhline{~--} & PL-FSK = PL-ISK & \cite{ AJLM2022}, this paper \\
\hhline{~--} & KSSC & \cite{ AJLM2022}, this paper \\
\hhline{~--} &  has CF-decomposition & \cite{FLRM2019} \\
\hhline{~--} &  unique positive equilibrium in each & \multirow{2}{*}{\cite{SHFE2012}, this paper} \\
\hhline{~~~} &   stoichiometric class ($\Rightarrow$ monostationary) & \\
\hhline{~--} & co-monostationary & This paper  \\
\hhline{~--} & absolutely complex balanced & \cite{FEIN1972}  \\
\hhline{~--} & not positive equilibria  & \multirow{2}{*}{\cite{LLMM2022}, this paper} \\
\hhline{~~~} & log-parametrized (not PLP) & \\
\hhline{~--} & not complex balanced equilibria  & \multirow{2}{*}{\cite{JOMT2022}, this paper} \\
\hhline{~~~} & log-parametrized (not CLP) & \\
\hhline{~--} & no ACR species & \cite{FOME2021}, this paper  \\
\hhline{~--} & LDC exists & \cite{NEML2019}  \\
\hhline{~--} & non-ACB LDC & \cite{JOMT2022} , this paper \\
\hhline{~--} & exponential stability of equilibria & This paper \\
\hline
\end{tabular}
\end{minipage}
\end{center}
\end{table}  

\section{Reaction network analysis of Anderies systems}\label{sec5}
As discussed in Section \ref{sec3}, the variation of the human intervention parameter in the Anderies et al. model led in the initial studies (\cite{DOA2018,FLRM2021}) to different families of kinetic representations as power law systems, which we collectively call \textbf{Anderies systems}. Each family is characterized by its common interaction function as specified by a kinetic order matrix. Two families, whose kinetic order matrices are given in Section \ref{sec3},  have been studied so far: members of the first family were shown in \cite{DOA2018} to be multistationary, while those of the latter displayed ACR in at least one species (\cite{FLRM2021}). In this Section, we extend these initial findings to a complete classification of Anderies systems based on their characterization as PLP-systems with a common flux space. We present new results on the multiplicity and co-multiplicity of their positive equilibria (relative to stoichiometric classes). The main tools used are the KSSC property and the availability of an LDC of any Anderies system.

\subsection{The KSSC property of Anderies systems}\label{sec5.1}

Since the underlying network of any Anderies system is $t$-minimal and the kinetics is interaction span surjective, the extended KSSC Theorem implies the coincidence of the kinetic and stoichiometric subspaces. A very important consequence of this property is that, for any  $(\mathscr{N}, K)$ with dynamically equivalent system $(\mathscr{N}^\#, K^\#)$, one has $\mathcal{S} = \mathcal{S}_\#$. As we will see in the succeeding sections, this allows the inference of important properties for the original system.

\subsection{Availability of a weakly reversible deficiency zero LDC}\label{sec5.2}

The LDC of an Anderies system is the following dynamically equivalent system first observed by D. Talabis in January 2019 (to date, unpublished):

$$
\left.
  \begin{array}{rcl}
A_1 + 2A_2 &\rightleftarrows &2A_1 + A_2 \\
A_2 &\rightleftarrows & A_3 \\
  \end{array}
 \right.
$$
Since $2A_1 + A_2=(A_1+A_2)+A_1 \rightarrow 2A_2 + A_1$, the second reactions have the same reaction vector. Being weakly reversible, it is $t$-minimal and hence, the KSSC property. Significantly, it has zero deficiency, and hence is absolutely complex balanced.

\subsection{The structure of the set of positive equilibria of Anderies systems} \label{sec5.3}
In this Section, we exploit the LDC to infer properties of Anderies systems. We briefly review concepts and results on log-parameterized (LP) systems before showing that any Anderies system is a PLP system. This allows the classification of Anderies systems via their 1-dimensional parameter subspace and the inference of multiplicity properties for each class.  We conclude by briefly discussing the finest independent decomposition of an Anderies system.

\subsubsection{Properties of LP systems}

F. Horn and R. Jackson \cite{HORNJACK1972} pioneered the study of LP systems in 1972 through their work on thermostatic mass action systems,  whose flux space is the stoichiometric subspace. The basic concepts are:
\begin{definition}\label{def:LP}
A kinetic system $(\mathscr{N},K)$ is of type 
\begin{enumerate}[label=(\roman*)]
    \item \textbf{PLP (positive equilibria log-parameterized)} if $E_+(\mathscr{N},K) \neq \emptyset$ and  $E_+(\mathscr{N},K) = \{ x\in \mathbb{R}^\mathscr{S}_{>0} \mid \log x - \log x^* \in (P_E)^\perp \}$, where $P_E$ is a subspace of $\mathbb{R}^\mathscr{S}$ and $x^*$ is a positive equilibrium.
    \item \textbf{CLP (complex-balanced equilibria log-parameterized)} if $Z_+(\mathscr{N},K) \neq \emptyset$ and $Z_+(\mathscr{N},K) = \{ x\in \mathbb{R}^\mathscr{S}_{>0} \mid \log x - \log x^* \in (P_Z)^\perp \}$, where $P_Z$ is a subspace of $\mathbb{R}^\mathscr{S}$ and $x^*$ is a complex-balanced equilibrium.
    \item \textbf{bi-LP} if it is of PLP and of CLP type, and $P_E=P_Z$. 
\end{enumerate}
$P_E$ and $P_Z$ are called \textbf{flux subspaces} of the system. 
\end{definition}
 The following proposition from \cite{JOMT2022} justifies the name ``parameter space" for their corresponding orthogonal complements:
\begin{proposition} (Proposition 3 of \cite{JOMT2022})
If $(\mathscr{N},K)$ is a chemical kinetic system of type PLP with flux subspace $P_E$ and reference point $x^* \in E_+(\mathscr{N},K)$, then the map $L_{x^*} : E_+(\mathscr{N},K) \rightarrow (P_E)^\perp$ given by $ L_{x^*} (x) = \log x - \log x^*$ is a bijection.
\end{proposition}
\noindent An analogous result holds for CLP systems. The final theorem from  \cite{JOMT2022} that we will use is:

\begin{theorem} (Theorem 4 of \cite{JOMT2022}) \label{thm:CLP}
Let $(\mathscr{N},K)$ be a CLP system with flux subspace $P_Z$ and reference point $x^* \in Z_+(\mathscr{N},K)$. Then $(\mathscr{N},K)$ is absolutely complex balanced if and only if $(\mathscr{N},K)$ is a bi-LP system. 
\end{theorem}

\subsubsection{Any Anderies system is a PLP system}

The following general proposition is the basis of our result:

\begin{proposition}
If a PLK system $(\mathscr{N},K)$ with $E_+(\mathscr{N},K)\neq \emptyset$ is dynamically equivalent to a deficiency zero PL-RDK system $(\mathscr{N}^\#, K^\#)$, then it is PLP with $P_E =\widetilde{\mathcal{S}}_\#$. 
\end{proposition}

\begin{proof}
In view of dynamic equivalence, $E_+(\mathscr{N}^\#, K^\#)= E_+(\mathscr{N},K)\neq \emptyset$. By Feinberg's ACB Theorem (see Remark \ref{rem:CB}), $(\mathscr{N}^\#, K^\#)$ is absolutely complex balanced. By Prop. 2.21 of M\"{u}ller and Regensburger \cite{MURE2012}, it is a CLP system. Applying Theorem \ref{thm:CLP}, we obtain the claim.
\end{proof}
Any Anderies system satisfies the assumptions of the proposition with $(\mathscr{N}^\#, K^\#)$ given by its LDC. A calculation of the 1-dimensional parameter space $(\widetilde{\mathcal{S}}_\#)^\perp$  provides the following basis:
\begin{equation}\label{eq:basis}
\{ v \}= \left\lbrace
\left( 
-1, \frac{p_2-p_1}{q_2-q_1}, \frac{p_2-p_1}{q_2-q_1}
\right)
\right\rbrace.
\end{equation}
We denote the ratio $\frac{p_2-p_1}{q_2-q_1}$ with $R$.

%\subsubsection{Classes of Anderies systems, equilibria multiplicity and network discordance}
We note that Anderies systems, like Schmitz systems, are conservative, closed systems of maximal rank ($s = 2, m = 3$). As a consequence, equilibria multiplicity is closely related to ACR properties. For the analysis of this relationship in the Section \ref{sec5.4}, we use the values of $R$ to introduce 3 classes of Anderies systems:

\begin{definition}
The set of Anderies systems with $R > 0$ $(R < 0)$ is denoted by $\bm{\textsf{AND}_>}$  $(\bm{\textsf{AND}_<})$. The set of Anderies systems with $R = 0$ is denoted by $\bm{\textsf{AND}_0}$.
\end{definition}

\subsubsection{The finest independent decomposition of an Anderies system}
For any independent decomposition, its length is less than the network's rank, implying that the finest such decomposition has at most 2 subnetworks. This is clearly given by the subnetwork $\mathscr{L}_1 \cup \mathscr{L}_2$ and the subnetwork $\mathscr{L}_3$, where $\mathscr{L}_1 =\{ A_1 +2A_2 \rightarrow 2A_1 + A_2\}$, $\mathscr{L}_2= \{ A_1+A_2 \rightarrow 2A_2 \}$, and $\mathscr{L}_3= \{ A_2 \leftrightarrows A_3 \}$.

\subsection{ACR analysis of Anderies systems}\label{sec5.4}
A big advantage of ACR analysis in LP systems is the availability of a necessary and sufficient condition for the property through the species hyperplance criterion (SHC).

\subsubsection{ACR analysis in LP systems}
The species hyperplane criterion for PLP systems in \cite{LLMM2022} states:

\begin{theorem} (Theorem 3.12, \cite{LLMM2022}) 
%Let $(\mathscr{N},K)$ be an LP system. 
If $(\mathscr{N},K)$ is a PLP system, then it has ACR is a species $S$ is and only if its parameter subspace $(P_E)^\perp$ is a subspace of the hyperplane $\{ x \in \mathbb{R}^\mathscr{S} \mid x_S = 0 \}$.
\end{theorem}

\noindent In the same study, the authors derive a simple procedure for assessing ACR in a PLP system in the following proposition:

\begin{proposition} (Prop. 4.1, \cite{LLMM2022}) 
Let $\{ v_1, \dots, v_E \}$ be a basis of the parameter subspace $(P_E)^\perp$ of a PLP system $(\mathscr{N},K)$. The system has ACR is species $S$ if and only if the coordinate corresponding to $S$ in each basis vector $v_{i,S}=0$ for each $i=1,\dots,E$. 
\end{proposition}

\subsubsection{ACR dichotomy among Anderies systems}
The coordinates of the basis vector (in Eq. \ref{eq:basis}) for  are all non-zero in the cases of $\textsf{AND}_>$ and $\textsf{AND}_<$, implying that in those cases, the systems do not have ACR in any of the species. On the other hand, for systems in $\textsf{AND}_0$, ACR holds for the species $A_2$ and $A_3$.

\subsubsection{ACR and equilibria multiplicity analysis in $\textsf{AND}_>$ systems}
In this Section, we show that the absence of ACR in these systems derives from its co-multistationarity property. Furthermore, we apply a result from M\"{u}ller and Regensburger \cite{MURE2012} to prove their multistationarity.

\begin{proposition}
Any $\textsf{AND}_>$ system is co-multistationary.
\end{proposition}

\begin{proof}
The power law approximation of an Anderies system is computed in \cite{DOA2018} as follows:
\begin{equation*} 
  \begin{array}{cl}
\dot{A}_1 &= k_1 A_1^{p_1}A_2^{q_1} - k_2 A_1^{p_2}A_2^{q_2} \\
\dot{A}_2 &= k_2 A_1^{p_2}A_2^{q_2} - k_1 A_1^{p_1}A_2^{q_1} - a_m A_2 + a_m \beta A_3   \\
\dot{A}_3 &= a_m A_2 - a_m \beta A_3.
  \end{array}
\end{equation*} 
Suppose $ A = \left( A_1, A_2, A_3 \right)$ and $ A' = \left( A_1', A_2', A_3' \right)$ are positive equilibria and $\alpha >0$. Note that $\mathcal{S}^\perp =\text{span } \{ (1,1,1) \}$. For the difference to lie in $\mathcal{S}^\perp$, from the ODE system, assuming $\beta = 1$, we have:
$$ A_2' - A_2 = A_3' - A_3 = \alpha.$$
Furthermore, $k_1 A_1^{p_1} A_2^{q_1} = k_2 A_1^{p_2} A_2^{q_2}$. Similarly, for $A'$. Subtracting the first from the second, we have
\begin{equation}\label{eq:AND>}
(A_1')^{p_1-p_2} - A_1^{p_1-p_2} = \dfrac{k_2}{k_1} \left[ (A_2')^{q_1-q_2} - A_2^{q_1-q_2} \right].
\end{equation}
For $\textsf{AND}_>$, we can assume that both $p_1 – p_2$ and $q_1 – q_2 > 0$. 
We next assume that for $a, b, y$ positive real numbers, $a > b$, the (formal)  equation holds (the intermediate terms just eliminate each other):
$$ a^y – b^y =(a-b) \left( a^{y-1}+a^{y-2}b + \cdots + b^{y-1} \right).$$
For convenience, write $\textsf{SUM} (a, b, y)$ for the second factor of the RHS of the previous equation. Substituting in Equation \ref{eq:AND>}, we get
$$
A_1' – A_1 = \dfrac{k_2 \left( A_2' – A_2 \right) \textsf{SUM} \left( A_2', A_2, q_1 -q_2 \right)}{k_1 \textsf{SUM}\left( A_1', A_1, p_1 - p_2 \right)}.
$$
The condition for $A_1' – A_1= \alpha$ is hence $$ \frac{k_2}{k_1} \textsf{SUM} \left( A_2', A_2, q_1 -q_2 \right)= \textsf{SUM}\left( A_1', A_1, p_1 - p_2 \right).$$
If we choose an $A_2$ , for any $\alpha$,  the LHS is just a positive real number. The RHS is a continuous, monotonically increasing function, so there is an $A_1$ such that the RHS takes on the value of the LHS (by the Intermediate Value Theorem). This gives two equilibria whose difference is $(\alpha,\alpha,\alpha)$ and hence co-multistationarity.
\end{proof}

 We will now use the LDC to derive equilibria multiplicity properties of the 3 classes. The following general proposition provides the basis for the analysis:
 \begin{proposition}
 If $(\mathscr{N},K)$ and $(\mathscr{N}^\#, K^\#)$  are dynamically equivalent with $\mathcal{S} = \mathcal{S}_\#$, then $(\mathscr{N},K)$ is multistationary if and only if $(\mathscr{N}^\#, K^\#)$ is multistationary.
 \end{proposition}
 \begin{proof}
 Since the positive equilibria sets (due to dynamical equivalence) and the stoichiometric classes (due to $\mathcal{S} = \mathcal{S}_\#$) are equal, we immediately obtain the equivalence.
 \end{proof}
This applies to an Anderies system and its LDC: since the identical kinetics is ISK, and both networks are $t$-minimal, both systems have KSSC, leading to $\mathcal{S} = \mathcal{S}_\#$.

We will first show that any system in $\textsf{AND}_>$ is multistationary. The generalized mass action systems (GMAS) of M\"{u}ller and Regensburger in their 2012 paper \cite{MURE2012} correspond to PL-RDK systems which are ISK (note that for cycle terminal systems, ISK = FSK, i.e. factor span surjective). For weakly reversible PL-FSK systems, they have the following sufficient condition for multistationarity (in terms of sign spaces):

\begin{proposition}(Proposition 3.2 of \cite{MURE2012})
If for a weakly reversible generalized mass action system with $\sigma(\mathcal{S}) \cap \sigma(\widetilde{\mathcal{S}})^\perp \neq \{ 0 \}$, then there is a stoichiometric class with more than one complex balanced equilibrium.
\end{proposition}
For the 1-dimensional subspace $(\widetilde{\mathcal{S}})^\perp$ of $\textsf{AND}_>$, we have  $\sigma(\widetilde{\mathcal{S}})^\perp = \{ (-,+,+), (+,-,-) \} $. Since $\mathcal{S}=\alpha (-1,1,0)+\beta(0,-1,1)=(-\alpha,\alpha-\beta,\beta)$, choosing $\alpha>\beta>0$ gives an element in $\mathcal{S}$ with  $\{ (-,+,+) \}$, verifying the non-empty intersection.

The application of CRNToolbox's Concordance Test confirms the new observation that the Anderies network is discordant. However, since the Anderies systems in \cite{DOA2018} and  \cite{FAML2021} are non-PL-NIK, we cannot directly derive this from their multistationarity.

\subsubsection{ACR and equilibria multiplicity analysis in $\textsf{AND}_0$ systems}

Since any $\textsf{AND}_0$ system has ACR in species $A_2$ and $A_3$, the following Proposition is an immediate Corollary of Proposition \ref{prop:ACRthenco-mono}  in Section \ref{sec4.2.2}:

\begin{proposition}
Any $\textsf{AND}_0$ system is co-monostationary.
\end{proposition}

We now show that monostationarity of $\textsf{AND}_0$ systems also derives from its ACR properties.

\begin{proposition}
Any $\textsf{AND}_0$ system is monostationary.
\end{proposition}

\begin{proof}
It is shown in \cite{FLRM2021} that when $p_1=p_2$ the equilibria set is given by
$$
E_+(\mathscr{N},K)= 
\left\lbrace
\left[ 
\begin{array}{c}
A_1 \\
A_2 \\
A_3 \\  
\end{array} 
\right] \in \mathbb{R}^\mathscr{S}_{>0}  \; \middle\vert \;
\
\begin{array}{ll}
A_2 &=
\left(
\dfrac{k_2}{k_1} \right)^{\frac{1}{q_1-q_2}}, \\
A_3 &=\dfrac{1}{\beta}\left(
\dfrac{k_2}{k_1} \right)^{\frac{1}{q_1-q_2}}, \text{ and}\\
A_1 &= A_0 - \left( 1+ \dfrac{1}{\beta} \right) \left(
\dfrac{k_2}{k_1} \right)^{\frac{1}{q_1-q_2}}
\end{array}
\right\rbrace,
$$
where $A_0=$ total conserved carbon at pre-industrial state. 
If there are two distinct equilibria in a stoichiometric class, they can differ only in $A_1$, since both $A_2$ and $A_3$ have ACR. One can set $\beta = 1$ (so that $A_2 = A_3$).  Hence, $A_1$ has two terms, $A_0$ and 2 times the value of $A_2$. However, for two elements in the same stoichiometric class, $A_0$ is the same (easily derived from the definition of conserved amount). This implies that there is at most one positive equilibrium in each stoichiometric class, i.e. the system is monostationary.
\end{proof}

\subsubsection{ACR and equilibria multiplicity analysis in $\textsf{AND}_<$ systems}\label{sec:AND<}

For $\textsf{AND}_<$ systems, we identify two subsets of injective systems, which are necessarily monostationary by applying the following result of Wiuf and Feliu \cite{WIUF2013}:

\begin{theorem}(\cite{WIUF2013,FELIU2013})\label{theorem:wiuffeliu}
The interaction network with power law kinetics and fixed kinetic orders is injective if and only if the determinant of $M^*$ is a non-zero homogeneous polynomial with all coefficients being positive or all being negative. 
\end{theorem}

In the above statement, the matrix $M^*$ is obtained by considering symbolic vectors $k=(k_1,\dots,k_m)$ and $z=(z_1,\dots,z_r)$ and letting $M=N \text{diag}(z) F \text{diag}(k)$, where $N$ is the stoichiometric matrix and $F$ is the kinetic order matrix of the PLK system. Let $\{ \omega^1,\dots,\omega^d \}$ be a basis of the left kernel of $N$ and $i_1,\dots,i_d$ be row indices. The $m \times m$ matrix $M^*$ is defined by replacing the $i_j$-th row of $M$ by $\omega^j$. The matrix $M^*$ is a symbolic matrix in $z_*$ and $k_*$.

Using the computational approach and Maple script provided by the authors in \cite{FELIU2013}, we obtain the determinant of $M^*$ for Anderies systems: 
$$\det (M^*) =-p_1 k_1 k_2 z_1 z_3-p_1 k_1 k_3 z_1 z_4+p_2 k_1 k_2 z_2 z_4+q_1 k_2 k_3 z_1 z_4-q_2 k_2 k_3 z_2 z_4.$$

\noindent Hence, for $p_1<0,p_2>0,q_1>0$, and $q_2<0$, all the terms are positive. For $p_1>0,p_2<0,q_1<0$, and $q_2>0$, all the terms are negative. In both cases, the networks are injective by Theorem \ref{theorem:wiuffeliu} and hence, monostationary. In all other cases, the systems are non-injective, which is a necessary condition for multistationarity.

\subsection{Exponential stability of equilibria}

The computation of the Jacobian matrix of $f$ at a positive steady and its eigenvalues is done in Maple. Each of the $s=2$ nonzero eigenvalues has negative real parts. Hence, all the nondegenerate steady states of both models are stable. For instance, for the kinetic orders, rate constants, and steady state values of the $\textsf{AND}_>$ system in \cite{DOA2018}, the non-zero eigenvalues are $\lambda=-0.8523,-0.101$. On the other hand, for the $\textsf{AND}_0$ system specified in \cite{FLRM2021}, the non-zero eigenvalues are $\lambda= -3.55\cdot 10^8, -0.0995$.

\subsection{Properties of Anderies systems: an overview}
Table \ref{table:anderies} provides an overview of the salient network and kinetic properties of the Anderies et al. model.

\begin{table}[!ht]
\begin{center}
\begin{minipage}{\textwidth}
\caption{Overview of network and kinetic properties of the Anderies model.}
\label{table:anderies}
\begin{tabular}{|c|l|c|}
\hline
\textbf{Property class} & \textbf{Anderies system} & \textbf{References} \\ 
\hline
\multirow{3}{*}{Network}    &   disconnected with 3 linkage classes &   \cite{DOA2018}  \\
\hhline{~--}  & non-weakly reversible & \cite{DOA2018} \\
\hhline{~--} (Digraph)& non-cycle terminal &    \cite{DOA2018} \\
\hhline{~--} &  $t$-minimal &   \cite{DOA2018} \\
\hline
\multirow{8}{*}{Network} & disconcordant    & This paper    \\
\hhline{~--} & conservative & This paper  \\
\hhline{~--} & non-monospecies & \cite{DOA2018} \\
\hhline{~--} &  deficiency = 1 & \cite{DOA2018} \\
\hhline{~--} (Stochiometry-related) &  positive dependent & \cite{DOA2018} \\
\hhline{~--} &  stoichiometric subspace is hyperplane & This paper \\
\hhline{~--} &  dependent linkage classes & \cite{DOA2018} \\
\hhline{~--} &  finest independent decomposition $\neq$  & \multirow{2}{*}{This paper} \\
\hhline{~~~} &  finest incidence independent decomposition & \\
\hline
\multirow{12}{*}{ Kinetic system}  & All: PL-RDK &  \cite{DOA2018}  \\
\hhline{~--} & All: non-PL-NIK & \cite{DOA2018,FLRM2021}  \\
\hhline{~--} & All: PL-ISK & This paper \\
\hhline{~--} & All: KSSC & This paper \\
\hhline{~--} & All: CF-decomposition. & \cite{DOA2018} \\
\hhline{~--} & All: not absolutely complex balanced & This paper  \\
\hhline{~--} & All: positive equilibria  & \multirow{2}{*}{This paper} \\
\hhline{~~~} & log-parametrized (PLP) & \\
\hhline{~--} & All: not complex balanced equilibria  & \multirow{2}{*}{\cite{DOA2018}} \\
\hhline{~~~} & log-parametrized (not CLP) & \\
\hhline{~--} & All: LDC exists & This paper  \\
\hhline{~--} & All: LDC is ACB & This paper \\
\hhline{~--} & All: exponential stability of equilibria & This paper \\
\hline
\multirow{2}{*}{ACR}  & $\textsf{AND}_>$ \& $\textsf{AND}_<$: No ACR &  This paper  \\
\hhline{~--} & $\textsf{AND}_0$: ACR in 2 species & \cite{DOA2018}, this paper  \\
\hline
\multirow{3}{*}{Equilibria multiplicity}  & $\textsf{AND}_>$ : multistationary &  \cite{DOA2018}, this paper    \\
\hhline{~--}  & $\textsf{AND}_0$: monostationary &  This paper  \\
\hhline{~--}  & $\textsf{AND}_<$: contains monostationary systems &  This paper  \\
\hline
\multirow{2}{*}{Equilibria co-multiplicity}  & $\textsf{AND}_>$ : co-multistationary &  This paper  \\
\hhline{~--}  & $\textsf{AND}_0$: co-monostationary &  This paper  \\
\hline
\end{tabular}
\end{minipage}
\end{center}
\end{table}

\section{Comparison of kinetic representations of an aggregated Schmitz model and Anderies systems} \label{sec6}

Clearly, some differences between Schmitz and Anderies systems may result from the difference in the number of species in the underlying networks. In this Section, we construct an aggregated Schmitz model with the same species as the Anderies systems and compare their kinetic representations. Among the systems compared are kinetic representations with a minimal number of differences.

\subsection{The aggregated Schmitz model}
We reduce the Schmitz model by aggregating carbon pools and adjusting the carbon transfers. We use the notation of Anderies et al. for easier comparison: $\{ M_5, M_6 \} \leftrightarrow A_1$, $\{M_2, M_3, M_4 \} \leftrightarrow A_3$, and $M_1=A_2$. The reduced model and its kinetic order matrix are shown below.
\begin{equation*}
\begin{aligned}[c]
R_1 &: A_1 \rightarrow A_2 \\
R_2 &: A_2 \rightarrow A_1 \\
R_3 &: A_2 \rightarrow A_3 \\
R_4 &: A_3 \rightarrow A_1 
\end{aligned}
\qquad \qquad
F=
\begin{blockarray}{lccc}
 & A_1 & A_2 & A_3   \\
\begin{block}{l[ccc]}
R_1 & 1 & 0 & 0 \\
R_2 & 0 & 0.36 & 0 \\
R_3 & 0 & 1 & 0 \\
R_4 & 0 & 0 & 9.8 \\
\end{block}
\end{blockarray}
\end{equation*}
The system is PL-NDK with a single NDK node $A_2$. The kinetic order 9.8 is the average of the kinetic orders $M_2 \rightarrow M_1$ (= 9.4) and $M_3 \rightarrow M_1$ in the original model.

Though the aggregated model is (as one would expect) not dynamically equivalent to the original one, an interesting result is that, except for some network numbers (number of complexes, number of reactions and rank) and the deficiency of the aggregated LDC, it shares all other network and kinetic properties of the original Schmitz system:

\begin{proposition}
The aggregated Schmitz system has the same network and kinetic properties of the Schmitz system in Table \ref{table:schmitz} Section \ref {sec4.6} except for the last one.
\end{proposition}
\begin{proof}
 We provide details only for some non-straight derivations. It is conservative since $(1,1,1)$ is a positive vector orthogonal to both basis vectors of $\mathcal{S}$. The stoichiometric subspace $\mathcal{S}$ is a hyperplane because rank  $s= 2$ and $m = 3$. The finest independent decomposition is $A_1 \rightleftarrows A_2 \cup A_2 \rightleftarrows A_3$ is also incidence independent. This is also the system’s CF-decomposition because deficiency = 0. It is CLP after Fontanil et al. \cite{FMF2021} and hence bi-LP due to zero deficiency.
\end{proof}

\begin{remark}
The aggregated linear conjugate of the Schmitz system is given by: $A_1 \rightleftarrows 2A_2$, $A_2 \rightleftarrows 2A_3$. This is now a deficiency zero PL-RDK system with 4 monospecies complexes and 2 linkage classes.
\end{remark}

\subsection{Comparison of kinetic representations of the aggregated Schmitz model with $\textsf{AND}_0$ LDCs}

It is also instructive to compare the properties of the aggregated Schmitz model with those of the deficiency zero representation of $\textsf{AND}_0$ systems. Table \ref{table:comparison} collects the remaining differences between the models\textemdash all other properties coincide.

\begin{table}[!ht]
\begin{center}
\begin{minipage}{\textwidth}
\caption{Comparison of the aggregated Schmitz model with $\textsf{AND}_0$ LDCs.}
\label{table:comparison}
\begin{tabular}{|c|c|c|}
\hline
\textbf{Property type} & \textbf{Aggregated Schmitz system} & \textbf{$\bm{\textsf{AND}_0}$ LDC} \\ 
\hline
\multirow{3}{*}{Network}    &   connected (1 linkage class) &   non-connected (2 linkage classes)   \\
\hhline{~--}  & 3 monomolecular complexes & 2 mono- + 2 bimolecular complexes \\
\hhline{~--} & concordant & discordant \\
\hline
\multirow{3}{*}{ Kinetic system}  & PL-NDK &  PL-RDK    \\
\hhline{~--} & PL-NIK & non-PL-NIK  \\
\hhline{~--} & No ACR in any species & ACR in 2 species \\
\hline
\end{tabular}
\end{minipage}
\end{center}
\end{table} 

These six characteristics constitute, in our view, the essential structural and kinetic differences between the Schmitz model and this class of Anderies systems. 
 
With respect to $\textsf{AND}_>$ systems, the ACR difference is replaced by two: monostationary/multistationary and co-monostationary/co-multistationary.   

\begin{remark}
For the LDC of an $\textsf{AND}_>$  system, the difference in ACR properties  is replaced by the differences montostationarity vs. multistationarity and co-monostationarity vs. co-multistationarity.
\end{remark}

\section{A kinetic representation of tCDR model with Anderies systems}\label{sec7}

In this Section, we illustrate the usefulness of the Anderies pre-industrial models as building blocks to form an analysis of a kinetic representation of a model of carbon dioxide removal (CDR). CDR methods, also called negative emission technologies (NETs), play an increasingly important role in strategies toward carbon neutrality and are/will be useful in addressing climate change. %The paper of Heck et al. \cite{HECK2016} looked into  terrestrial carbon dioxide removal (tCDR), which includes reforestation methods.

Heck et al. \cite{HECK2016} investigated the dynamics of Earth's carbon cycle when climate engineering via terrestrial carbon dioxide removal (tCDR) is considered as human intervention. In this intervention, terrestrial carbon is sequestered and permanently stored in a carbon engineering sink. The conceptual model was built upon the model of Anderies et al. \cite{AND2013}. The latter was modified to represent better the empirically observed and simulated Earth system carbon dynamics. In addition, the model was extended to integrate a societal management feedback loop that attempts to mimic international policies on climate change. 

\begin{figure}[t]%
\centering
\includegraphics[width=0.8\textwidth]{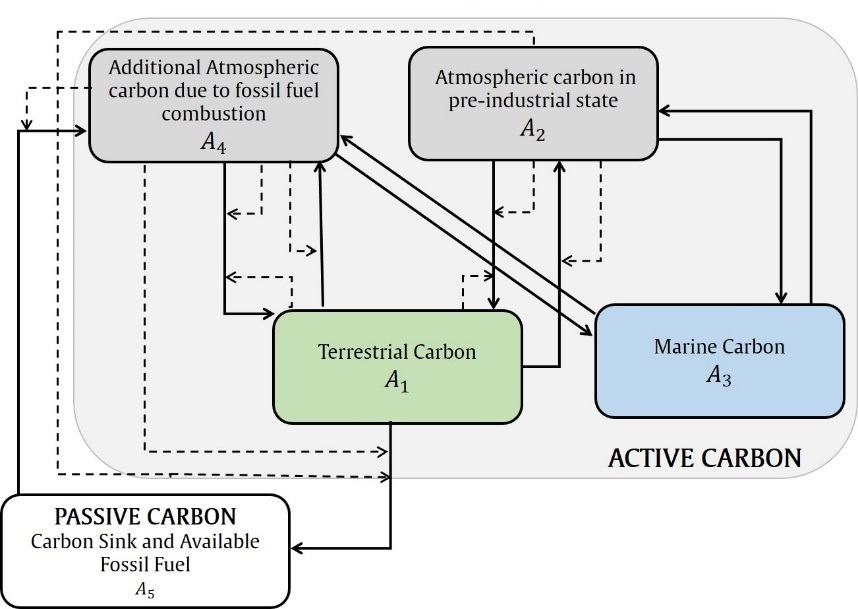}
\caption{Biochemical map of Heck et al.'s global carbon cycle model}\label{heck}
\end{figure}

The CRN-based analysis of the dynamics of the system begins by generating the network of reactions from the interactions summarized in the biochemical map in Figure \ref{heck}. The model considers pooling the geological carbon pool and the new sink to form a passive carbon pool ($A_5$) and decoupling the atmospheric carbon into two nodes: atmospheric carbon in the pre-industrial state ($A_2$) and additional atmospheric carbon due to fossil fuel use ($A_4$). The complete set of reactions is given below.

\begin{equation*}
\begin{aligned}[c]
R_1 &: A_1 +2A_2 \rightarrow 2A_1+ A_2 \\
R_2 &: A_1+A_2 \rightarrow 2A_2 \\
R_3 &: A_2 \rightarrow A_3 \\
R_4 &: A_3 \rightarrow A_2 \\
R_5 &: A_4+A_5 \rightarrow 2A_4 
\end{aligned}
\qquad \qquad
\begin{aligned}[c]
R_6 &: A_1 +2A_4 \rightarrow 2A_1+ A_4 \\
R_7 &: A_1+A_4 \rightarrow 2A_4 \\
R_8 &: A_4 \rightarrow A_3 \\
R_9 &: A_3 \rightarrow A_4 \\
R_{10} &: A_1+A_2+A_4 \rightarrow A_5 +A_2 + A_4 
\end{aligned}
\end{equation*}

It can be easily verified using CRNToolbox \cite{CRNTool} that the network is positive dependent (i.e., there are rate constants such that the system has a positive equilibrium), conservative, and closed with maximal rank ($m = 5, s = 4$). The corresponding approximated power law rate functions of the reactions are encoded in the following kinetic order matrix, indicating that the system is clearly non-PL-NIK. 
\begin{equation*}
F=
\begin{blockarray}{lccccc}
 & A_1 & A_2 & A_3 & A_4 & A_5   \\
\begin{block}{l[ccccc]}
R_1 & 199.75 & -86.03 & 0 & 0 & 0 \\
R_2 & 159.84 & -63.32 & 0 & 0 & 0 \\
R_3 & 0 & 1 & 0 & 0 & 0\\
R_4 & 0 & 0 & 1 & 0 & 0 \\
R_5 & 0 & 0 & 0 & 1 & 1.54 \\
R_6 & -43.80 & 0 & 0 & 21.42 & 0 \\
R_7 & -56.13& 0 & 0 & 22.19 & 0 \\
R_8 & 0 & 0 & 0 & 1 & 0 \\
R_9 & 0 & 0 & 1 & 0 & 0 \\
R_{10} & 1 & 4.44 & 0 & 11.52 & 0 \\
\end{block}
\end{blockarray}
\end{equation*}
There are two Anderies subnetworks, namely $\mathscr{A}_1 = \{R_1, R_2, R_3,R_4 \}$ and $\mathscr{A}_2 = \{ R_6, R_7, R_8,R_9 \}$. These subnetworks can be used to find a dynamically equivalent system of lower deficiency; that is, $\delta= 2$ instead of the original $\delta=4$. The following lower deficiency kinetic realization can be obtained using the similar observation expressed in Section \ref{sec5.2}:

\begin{equation*}
\begin{aligned}[c]
R_1, R_2' &: A_1 +2A_2 \rightleftarrows 2A_1+ A_2 \\
R_3, R_4  &: A_2 \rightleftarrows A_3 \\
R_5 &: A_4+A_5 \rightarrow 2A_4 
\end{aligned}
\qquad \qquad
\begin{aligned}[c]
R_6, R_7' &: A_1 +2A_4 \rightleftarrows 2A_1+ A_4 \\
R_8, R_9  &: A_3 \rightleftarrows A_4 \\
R_{10} &: A_1+A_2+A_4 \rightarrow A_5 +A_2 + A_4 
\end{aligned}
\end{equation*}

In terms of dynamical behavior, the aforementioned Anderies subsystems belong to the class $\textsf{AND}_<$, specifically to its non-injective subset (see Section \ref{sec:AND<}). This finding is consistent with the earlier result of Hernandez et al. \cite{HERMENDLR2022}, that the system has the capacity for multistationarity. Multistationarity in this context implies that there may exist ``tipping points" beyond which a return to the previous state may be difficult or prolonged. Nevertheless, with the foreknowledge that multistationarity may exist, the numerical search for tipping points may be guided. If a tipping point is identified, appropriate actions may then be set to avoid exceeding it.

\section{Summary and Outlook}\label{sec8}

The global carbon cycle accounts for the different pools where carbon is stored (i.e., atmosphere, ocean, land, and geological or fossil fuel pools) and the processes which transfer carbon from one reservoir to another. The pre-industrial state, where there is no mass transfer of carbon from the fossil fuel pool to the atmospheric carbon pool, is an important aspect of studies related to climate change as it serves as a reference for a roughly balanced and desirable state.  

In this work, we conducted a comparative analysis of the power law kinetic representations of the carbon cycle models of Schmitz \cite{SCHM2002} and Anderies et al. \cite{AND2013} at pre-industrial state. With the methods and techniques found in chemical reactions network theory, we were able to expand the analysis of the kinetic representations of the two models and identify the similarities and differences in their network and kinetic properties in relation to model construction assumptions (described in Section \ref{sec3}).

Along with previous results, the novel results established in this paper are consolidated and summarized in Table \ref{table:schmitz} and Table \ref{table:anderies} to easily compare the structural and dynamic properties of the two systems. As some differences between Schmitz and Anderies systems may result from the difference in the number of species in the underlying networks, we constructed an aggregated Schmitz model with the same species as the Anderies system. In Section \ref{sec6}, it was shown that an aggregated Schmitz model has the same structural and kinetic properties (with one exception) as a Schmitz system. Moreover, the comparison of an aggregated Schmitz system with the dynamically equivalent LDC of an Anderies system showed differences in only three structural and three kinetic properties (as summarized in Table \ref{table:comparison}). These contrasts may be viewed as the essential properties resulting from the different hypotheses underlying the Schmitz and Anderies et al. models (such as their respective non-isothermal and isothermal assumptions).

We also highlight that some of the new results observed here are derived from more general propositions about conservative, closed kinetic systems of maximal rank (see Section \ref{sec4.2.2}), of which both Schmitz and Anderies systems are examples. These propositions contribute to the mathematical theory of power-law kinetic systems, which may be applicable in the analysis of other biological systems.

Our analysis of a kinetic representation of the tCDR model of Heck et al. \cite{HECK2016}  revealed that two of its subsystems are pre-industrial Anderies systems. We have shown that these Anderies subnetworks can be used to find a dynamically equivalent system of lower deficiency. Moreover, both of them belong to the class of non-injective subset of $\textsf{AND}_<$. This finding agrees with the earlier result of Hernandez et al. \cite{HERMENDLR2022} that the system has the capacity to admit multiple steady states. 

We also note that the concept of ``planetary boundaries" introduced in the Anderies et al. model has had a substantial impact on the global sustainability community (see, for example, the paper of Steffen et al. \cite{Steffen2015}, with to date over 9600 citations!). As part of our ongoing research, we are currently working on kinetic representations of further CDR methods such as direct air capture (DAC) and ocean fertilization (OF), which also are based on Anderies building blocks. As Tan et al. \cite{TAN2022} have pointed out, it is also important to optimize combinations or ``portfolios" of NETs. We can address this problem with poly-PL kinetic systems, i.e. sums of power law systems, which we have previously used for the analysis of Hill-type systems which are prevalent in biochemical processes \cite{HEME2021}.

%\bmhead{Supplementary information}

%\bmhead{Acknowledgments}

\section*{Declarations}
The authors declare no conflicts of interests.

%\newpage
\bibliographystyle{abbrv}
\bibliography{codysVrev1-arXiv}% common bib file

\newpage

\appendix

\section{Nomenclature} \label{append:nomenclature}

\subsection{List of important symbols}
\begin{table}[h]
\begin{tabular}{ll}
$(\mathscr{N},K)$ & Chemical kinetic system \\
$\delta$    &   Deficiency  \\
$F$ &   Kinetic order matrix    \\
%$Y$    &   Map of complexes    \\
$m$ & Number of species \\
$n$ & Number of complexes \\
$n_r$ & Number of reactant complexes \\
$r$ & Number of reactions \\
$\ell$ & Number of linkage classes \\
$s\ell$ & Number of strong linkage classes \\
$t$ & Number of terminal strong linkage classes \\
$s$ &   Rank of a CRN or $\dim (\mathcal{S})$   \\
$Z_+$ & Set of complex balanced equilibria \\
$E_+$   &   Set of positive equilibria  \\
$f$ &   Species formation rate function \\
$\mathcal{S}$   &   Stoichiometric subspace \\
\end{tabular}
\end{table}

\subsection{Abbreviations}

\begin{table}[ht!]
\begin{tabular}{ll}
ACB & Absolutely complex balanced \\
ACR & Absolute concentration robustness \\
CLP & Complex balanced equilibria log parametrized \\
CF & Complex factorizable \\
CRN &   Chemical Reactions Network  \\
CRNT & Chemical Reactions Network Theory    \\
FSK & Factor span surjective kinetics \\
ISK &  Interaction span surjective kinetics \\
KSSC & Kinetic/Stoichiometric Subspace Coincidence \\
LDC & Low-deficiency complement \\
LP & Log parametrized \\
NF & Non-complex factorizable \\
ODE &   Ordinary differential equation  \\
PLK & Power law kinetics  \\
PLP & positive equilibria log parametrized \\
PL-NDK  &   Power law with non-reactant-determined kinetics \\
PL-NIK & Power law with non-inhibitory kinetics \\
PL-RDK  &   Power law reactant-determined kinetics  \\
\end{tabular}
\end{table}

\end{document}